\documentclass[12pt,oneside,english]{amsart}
\usepackage[T1]{fontenc}
\usepackage[latin9]{inputenc}
\usepackage{geometry}
\usepackage{color}
\usepackage{amstext}
\usepackage{amsthm}
\usepackage{amssymb}

\usepackage{amsmath}
\usepackage{epsfig} 
\usepackage{graphics} 
\usepackage{amsfonts}
\usepackage{amssymb}
\usepackage{graphicx}
\usepackage{enumerate}
\usepackage{soul}
\usepackage[unicode=true,pdfusetitle,
 bookmarks=true,bookmarksnumbered=false,bookmarksopen=false,
 breaklinks=false,pdfborder={0 0 1},backref=false,colorlinks=false]
 {hyperref}

\makeatletter
\theoremstyle{plain}
\newtheorem{thm}{\protect\theoremname}
\theoremstyle{definition}
\newtheorem{defn}[thm]{\protect\definitionname}
\theoremstyle{plain}
\newtheorem{prop}[thm]{\protect\propositionname}
\theoremstyle{plain}
\newtheorem{lem}[thm]{\protect\lemmaname}
\theoremstyle{plain}
\newtheorem{rem}[thm]{\protect\remarkname}
\theoremstyle{plain}
\newtheorem{cor}[thm]{\protect\corolname}
\makeatother

\usepackage{babel}
\providecommand{\definitionname}{Definition}
\providecommand{\propositionname}{Proposition}
\providecommand{\theoremname}{Theorem}
\providecommand{\lemmaname}{Lemma}
\providecommand{\remarkname}{Remark}
\providecommand{\corolname}{Corollary}
\newcommand{\sm}{\smallskip}
\newcommand{\md}{\medskip}
\newcommand{\bg}{\bigskip}
\newcommand{\nd}{\noindent}
\newcommand{\RR}{\mathbb{R}}
\newcommand{\dis}{\displaystyle}
\newcommand{\de}{{\rm d}}
\newcommand{\beq}{\begin{equation}}
\newcommand{\eeq}{\end{equation}}

\newcommand{\eps}{\varepsilon}

\begin{document}
\pagestyle{plain}
\title{wavefront solutions for reaction-diffusion-convection models with
accumulation term and aggregative movements}
\author{Marco Cantarini$^{A}$, Cristina Marcelli$^{B}$, Francesca Papalini$^{B}$}
\maketitle

\noindent
\small\textit{$\,^{A}$Dipartimento di Matematica ed Informatica, Universit\`a degli Studi di Perugia, Via Vanvitelli 1, Perugia I-06123, Italy. \\
$\,^{B}$Dipartimento di Ingegneria Industriale e Scienze Matematiche, Universit\`a Politecnica delle Marche, Via Brecce Bianche 12, Ancona I-60131, Italy.}\\

\small\textit{email: marco.cantarini@unipg.it,
c.marcelli@staff.univpm.it,
f.papalini@staff.univpm.it}\\

\begin{abstract}
In this paper we analyze the wavefront solutions of parabolic partial differential
equations of the type
\[
g(u)u_{\tau}+f(u)u_{x}=\left(D(u)u_{x}\right)_{x}+\rho(u),\quad u\left(\tau,x\right)\in[0,1]
\]
where the reaction term \(\rho\) is of monostable-type.  
We allow the diffusivity \(D\) and the accumulation term \(g\) to have a finite number of changes of sign.

We provide an existence result of travelling wave solutions (t.w.s.) together with an estimate of the threshold wave speed.
Finally, we classify the t.w.s. between classical and sharp ones. 
\end{abstract}

\maketitle

\bigskip

{\bf AMS} Subject  Classifications: 35K57, 35K65, 35C07, 35K55,   34B40, 34B16, 92D25

\medskip


\section{Introduction}
In this paper we study the existence and the properties of travelling wave solutions (t.w.s.) for the following reaction-diffusion-convection equation
\begin{equation}
g(v)v_{\tau}+f(v)v_{x}=\left(D(v)v_{x}\right)_{x}+\rho(v), \quad v\left(\tau,x\right)\in[0,1]\label{eq:main}
\end{equation}
where 
\begin{equation}
g,f,\rho,D\in C\left([0,1]\right),\,D\in C^{1}\left(0,1\right),\label{ip:funz}
\end{equation}
\begin{equation}
\rho(v)>0,\,\text{ for every } v\in(0,1), \quad \rho(0)=\rho(1)=0,\label{eq:sign rho}
\end{equation}
arising in several physical and biological phenomena. Owing to the various applications, the study of reaction-diffusion-convection equations has been widely developed, but in general for equations in which \(g(v)\equiv 1\).

On the other hand, the presence of the accumulation term  $g$  allows to consider relevant physical phenomena, such as thermal processes when the heat capacity of the medium depends on temperature and in the theory of filtration of a fluid in a porous media (see \cite{CDDSV,D}).

\sm
Our interest is addressed to the investigation of t.w.s., that is solutions of the type $v(\tau,x):=u(x-c\tau)$, where \(c\) is the wave speed.  Such functions represent a relevant class of solutions of the equation \eqref{eq:main}, since 
at least in simple models the solution of initial-boundary value problem for the differential equations converges, for large times and in a some specific sense, to a profile of a t.w.s. (see, e.g.  \cite{AV,DT,KR}).  

\sm
 Notice that a t.w.s. of \eqref{eq:main}  in an interval \((a,b)\subset \mathbb{R}\)  is a solution
of the following second-order (possibly singular) equation
\begin{equation}
\left(D\left(u\right(t))u^{\prime}(t)\right)^{\prime}+\left(cg\left(u(t)\right)-f(u(t))\right)u^{\prime}(t)+\rho(u(t))=0 \quad \text{ for every } t\in (a,b)\label{eq:E}
\end{equation}
where ${}^{\prime}$ means that we differentiate with respect $t:=x-c\tau$,
connecting the zeros of the reaction term \(\rho\), that is the equilibria of the equation; hence \(u\) satisfies the 
 boundary conditions  \(u(a^+)=1\) and \(u(b^-)=0.\)

\sm 

The literature concerning t.w.s. for reaction-diffusion-convection equations is very wide.
 In the case $g(u)\equiv 1$, that is, the accumulation term $g$ is not significant in the model, and the reaction term is such that $\rho(u)>0,\,u\in (0,1)$ (the so-called monostable case), it is well known that  there exists a threshold minimal speed $c^{*}$, which can be explicitly estimated, such that if $c\geq c^{*}$, then the model \eqref{eq:main} admits t.w.s., and also the converse implication holds true. 

In general, if $c > c^{*}$,  the t.w.s. are defined and continuously differentiable on the whole real line. Instead, if $c=c^{*}$, the situation is more delicate: the t.w.s. is again smooth on the real line if the diffusion $D$ does not vanish at the equilibria $0,1$ (non degenerate case). Otherwise,  the t.w.s. with \(c=c^*\) can reach one/both the equilibria in a finite time with a non-zero slope and the dynamic is said to exhibit the phenomenon of finite speed of propagation and/or finite speed of saturation. For some references to these facts and more informations see, e.g., \cite{ACM,GM,L,MM,SM}. 

\sm
In the recent paper \cite{CaMaPa}, we proved that such a type of the aformentioned results can be achieved also if the accumulation $g$ is a continuous function not necessarily constant neither positive, so that 
the equation  presents various types of degeneracies, since both \(D\) and \(g\) can vanish, even simultaneously. Nevertheless, also in this case there exists a threshold minimum wave speed and it is possible to classify the emerging t.w.s. (see \cite[Theorem 14 and 16]{CaMaPa}).

\sm

In recent years, an increasing interest has been addressed to the investigation of aggregative-diffusive models, in which the term \(D\) can have changes of sign (see  \cite{BCM,BCM1,GK,GK2,MM,MMM}). Of course, the sign of \(D\)  can influence 
 the existence of t.w.s. and their regularity at the points where \(D\) vanishes. 
Therefore, the natural question arises is if it is possible to extend the results obtained in \cite{CaMaPa} in the case of positive diffusive terms (but with a non-constant accumulation term) to the more general framework of functions \(D\) having, at the most, a finite number of changes of sign. This is just the aim of this paper, in which we deal with equation \eqref{eq:E} under the assumptions \eqref{ip:funz}, \eqref{eq:sign rho} and 
\beq \label{ip:Dfinite}
D_{0}:=\left\{ u\in\left(0,1\right):\,D\left(u\right)=0\right\}  \ \text{ if finite, possibly empty.}
\eeq
Equations as \eqref{eq:main} with chainging-sign \(g\) and \(D\) arise, for instance, in the study of t.w.s. for the telegraph equation (see \cite{GK3}).

\sm
Our main result is Theorem \ref{t:main2} which asserts that under certain assumptions on the sign on the integral function of \(g\), there exists a threshold value \(\hat c\) such that if \(c>\hat c\) equation \eqref{eq:E} admits t.w.s., whereas they do not exist for \(c<\hat c\). Moreover, we also provide an estimate for the value \(\hat c\) which generalizes the known results for equations that are particular cases of \eqref{eq:E}. Finally, Proposition \ref{p:class} concerns the classification of the t.w.s.

\sm
We underline that this context includes also the case in which \(D\) does not change sign, but it may vanish somewhere in \((0,1)\) (see Example 3).

\sm
Finally,
we point out that in the case of changing-sign diffusivities, the existence of t.w.s. which are smooth when they assume values between the equilibria is no more ensured when \(c=\hat c\). We discuss this fact in Remarks \ref{r:cstar} and \ref{r:cstar-2}. 
However, in the particular case in which \(D\) has an unique change of sign, from positive to negative, this phenomenon does not occur and we can state the existence of t.w.s. for \(c=\hat c\) too (see Corollary \ref{c:one}).

\md
 The general technique we adopt consists in using the results of [7] in the intervals where D does not change sign and then to find conditions ensuring that the solutions  found in each interval can be glued together in a regular way. This will have consequences  on the admissible wave speeds.
 
\md 
The article is organized as follows. In Section 2 we recall some preliminary definitions and results; in Section 3 we show how the existence of t.w.s. is equivalent to  the solvability of a first order singular problem which is investigated  in Section 4. Section 5 is devoted to the existence/non-existence of t.w.s. and in Section 6 we provide the classification of t.w.s. Finally, we provide some simple examples of equations \eqref{eq:main}, discussing the existence of t.w.s, the estimate of \(\hat c\) and the classification of the t.w.s.

\section{Preliminary results}\label{s:preliminary}

\bg
Troughout the paper, we always assume the validity of conditions 
\eqref{ip:funz}, \eqref{eq:sign rho} and \eqref{ip:Dfinite}.

\md 
First of all, we give the definition of t.w.s.

\begin{defn}
\label{d:tws}
A travelling wave solution (t.w.s. for short) of (\ref{eq:main}) is a function
$u\in C^{1}(a,b)$, with $(a,b)\subseteq\mathbb{R}$, such that $u(t)\in [0,1]$,
$D(u(\cdot))u^{\prime}(\cdot)\in C^{1}\left(a,b\right)$, satisfying equation (\ref{eq:E})
in $(a,b)$ and such that 
\begin{equation}
u\left(a^{+}\right)=1,\,u\left(b^{-}\right)=0\label{eq:as}
\end{equation}

\begin{equation}
\lim_{t\rightarrow a^{+}}D\left(u(t)\right)u^{\prime}(t)=\lim_{t\rightarrow b^{-}}D\left(u(t)\right)u^{\prime}(t)=0.\label{eq:limsh}
\end{equation}
\end{defn}

In what follows, we assume, without restriction, that the t.w.s. are always defined in their maximal existence interval, that is the maximal interval \((a,b)\), bounded or unbounded,  in which the previous conditions are satified. 

\sm
 When
 \((a,b)=\RR\) then condition \eqref{eq:limsh} is automatically satified, as the following result states.

\md

\begin{prop}
Let $u$ be a solution of (\ref{eq:E}) for some $c\in\mathbb{R}$,
satisfying (\ref{eq:as}). Then, if $a=-\infty$ we have $\dis\lim_{t\to -\infty}D\left(u(t)\right)u^{\prime}(t)=0$
and if $b=+\infty$ we have $\dis\lim_{t\to +\infty}D\left(u(t)\right)u^{\prime}(t)=0$.
\end{prop}

\begin{proof}
Let \(D_0:=\{ u\in (0,1): D(u)=0\}\). If \(D_0=\emptyset\) then the result has been proved in \cite[Proposition 3]{CaMaPa}. So, assume now \(D_0\) is nonempty, but finite, say
\( D_0=\{ u_1, \cdots, u_n\}.\)

Let us consider the case \(b=+\infty\) (the case \(a=-\infty\) is analogous).

Assume $b=+\infty$. Put  $T:=\inf\left\{ \xi:u\left(s\right)\leq u_{1},\,\forall s\in\left(\xi,+\infty\right)\right\} $, we have \(T>-\infty\) and \(u(T)=u_1\), so \(D(u(T))=0\). For any $t>T$, 
integrating \eqref{eq:E} in $\left[T,t\right]$ we
obtain
\[
 D\left(u(t)\right)u^{\prime}(t)=\int_{u_1}^{u(t)} (f(s)-cg(s) ) \de s -\int_{T}^{t}\rho\left(u(s)\right)\de s.
\]

Since \(\rho\) is a positive function, 
there exists (finite or not) the limit $\dis\lim_{t\to +\infty}\int_{T}^{t}\rho\left(u(s)\right) \de s$, hence there exists also the limit
\[
\lim_{t\rightarrow+\infty}D\left(u(t)\right)u^{\prime}(t)=:\lambda\in\left[-\infty,+\infty\right).
\]
Since $D$ is bounded and has constant sign in \([T,+\infty)\), if  $\lambda\neq0$ then there exists (finite or not) also the limit
\[
\lim_{t\rightarrow+\infty}u^{\prime}(t)=\lim_{t\rightarrow+\infty}\frac{D\left(u(t)\right)u^{\prime}(t)}{D\left(u(t)\right)}\ne 0,
\] in contradiction with the boundedness of $u$. So, we derive that
$\lambda=0$. 
\end{proof}

\bg

The next Lemma provides a necessary condition for the admissible wave speeds, which will be used in the following proposition.

\begin{lem}
\label{prop:segno int} If there exists a  t.w.s.  $u$  in \((a,b)\), then 
\beq
\int_{0}^{1}\left[cg\left(s\right)-f(s)\right] \de s>0.
\label{eq:intpos} \eeq
\end{lem}

\begin{proof}
Integrating (\ref{eq:E}) in $(a,b)$,  
by (\ref{eq:limsh}), we get 
\[
0<\int_{a}^{b}\rho\left(u(t)\right)\de t=-\int_{a}^{b}\left[cg\left(u(t)\right)-f\left(u(t)\right)\right]u^{\prime}(t)\de t=\int_{0}^{1}\left[cg\left(s\right)-f(s)\right]\de s.
\]
\end{proof}

The following result concerns the monotonicity property of the t.w.s., which is the key tool in order to reduce \eqref{eq:E} to a first order equation.

\begin{prop}
\label{prop:monotone u} If $u$ is a t.w.s. in \((a,b)\). Then,   the set 
\[ I_u:=\{t \in (a,b): \ 0<u(t)<1\}\]
is an interval and  \(u'(t)<0\) whenever \(t\in I_u\).

\end{prop}

\begin{proof}
We divide the proof into some steps.

\bg\nd
{\em Claim 1}:\ if \(u(t_0)=0\) for some \(t_0\in (a,b)\), then \(u(t)=0\) for every \(t\in (t_0,b)\); similarly,  if \(u(t_0)=1\) for some \(t_1\in (a,b)\), then \(u(t)=1\) for every \(t\in (a,t_1)\).

\md
Indeed, if there exists some $t_{0}\in\left(a,b\right)$ such that
$u\left(t_{0}\right)=0$ then $u^{\prime}\left(t_{0}\right)=0$ and
so, integrating (\ref{eq:E}) in \((t_0,b)\) we get
\[
\int_{t_{0}}^{b}\left(D\left(u(t)\right)u^{\prime}(t)\right)^{\prime}\de t+\int_{t_{0}}^{b}\left[cg\left(u(t)\right)-f\left(u(t)\right)\right]u^{\prime}(t)\de t+\int_{t_{0}}^{b}\rho\left(u(t)\right)\de t=0.
\]
Due to conditions \eqref{eq:as} and \eqref{eq:limsh},  the first and the second integral are null. So, since \(\rho\) is positive in \((0,1)\), we derive \(u(t)=0\) for every \(t\in (t_0,b)\). The second  statement of the claim is analogous.
 
\bg
{\em Claim 2}: 
\ if $0<u(t_0)<1$, \ $u'(t_0)=0\) and \(D(u(t_0))>0\) [resp. \(D(u(t_0))<0\)], then \(t_0\) is a point of proper local maximum [resp. minimum]. for the function \(u\).
\md

Let \(t_0\) be such that  $0<u(t_0)<1$, \ $u'(t_0)=0\) and \(D(u(t_0))>0\) (the
case $D\left(u\left(t_{0}\right)\right)<0$ is analogous).  From equation
(\ref{eq:E}) we get
\[
\left.\left(D\left(u(t)\right)u^{\prime}(t)\right)^{\prime}\right|_{t=t_{0}}=-\rho\left(u\left(t_{0}\right)\right)<0
\]
so the function $\left(D\circ u\right)u^{\prime}$ is strictly decreasing in
a neighborhood of $t_{0}$ and it vanishes at $t_{0}$. From the sign
of $D$ we deduce that $t_{0}$ is a proper local maximum point for
the function $u$.

\bg
{\em Claim 3}: 
\ if $0<u(t)<1$ and $u'(t)=0\), then  $D\left(u\left(t\right)\right)=0$.

\md
Assume, by contradiction, that for some \(t_0\) we have \(0<u(t_0)<1\), \(u'(t_0)=0\) and \(D(u(t_0))>0\). By Claim 2 the function \(u\) has a proper maximum at \(t_0\), so   taking account of the boundary datum 
 $u(a^+)=1$, we deduce the existence of a point \(\tau_0<t_0\)  such that \(u(t)\ge u(\tau_0)\) for every \(t \le t_0\),  with  \(u(\tau_0)<u(t_0)\). Again by Claim 2 we infer \(D(u(\tau_0))\le 0\). Hence, by the continuity of the function \(D\) in \([u(\tau_0),u(t_0)]\), a value \(u^*\) exists in \([u(\tau_0),u(t_0))\) such that \(D(u^*)=0\). Moreover, by the continuity of the function \(u\)
in \([\tau_0,t_0]\) there exists a point \(t^*\in [\tau_0,t_0)\) such that 
\(u(t^*)=u^*<u(t_0)\). Finally, taking account of the boundary datum \(u(b^-)=0\) we get the existence of a point \(T^*>t_0\) such that \(u(T^*)=u(t^*)=u^*\). Therefore,  
integrating (\ref{eq:E}) in $[t^*,T^*]$ we
get
\[\begin{array}{ll}
0 & =\dis\int_{t^*}^{T^*}(D\left(u(t)\right)u^{\prime}(t))'\de t+\dis \int_{t^*}^{T^*}\!\! [cg\left(u(t)\right)-f\left(u(t)\right)]u^{\prime}(t)\de t+\int_{t^*}^{T^*} \!\! \rho\left(u(t)\right)\de t \\ & = \dis\int_{t^*}^{T^*}\!\! \rho\left(u(t)\right)\de t>0,\end{array}
\]
a contradiction.

%
%
%
%

\bg
{\em Claim 4}: $u^{\prime}\left(t\right)\leq0 \) for every \(  t\in\left(a,b\right)\).

\md
Assume by contradiction that $u^{\prime}\left(t_{0}\right)>0$
for some $t_{0}\in\left(a,b\right)$. Let $\left(t_{1},t_{2}\right)$
be the largest interval containing $t_{0}$ such that $u^{\prime}\left(t\right)>0$
for every $t\in\left(t_{1},t_{2}\right)$. Of course, \(a<t_1<t_2<b\), hence 
\(u'(t_1)=u'(t_2)=0\). Let us now prove that \(u(t_1)=0\) and \(u(t_2)=1\). Indeed, if \(u(t_1)>0\), by virtue of Claim 3 we get \(D(u(t_1))=0\); moreover, by the boundary datum \(u(b^+)=0\) we get the existence of a point \(\tau_1>t_1\) such that \(u(\tau_1)=u(t_1)\). Therefore, integrating equation \eqref{eq:E} in \([t_1,\tau_1]\) we again obtain 
the contradiction \(\dis\int_{t_1}^{\tau_1} \rho(u(t))  \de t =0\). Then, necessarily \(u(t_1)=0\). Similarly we can prove that \(u(t_2)=1\).

\md
 Finally, integrayting again equation \eqref{eq:E} in  \([t_1,t_2]\) we get
 \[ 0=\int_{t_1}^{t_2} [cg(u(t))-f(u(t))]u'(t) \de t + \int_{t_1}^{t_2} \rho(u(t)) \de t = \int_0^1 [cg(s)-f(s)] \de s + \int_{t_1}^{t_2} \rho(u(t)) \de t,\]
which is a contradiction by \eqref{eq:intpos}, since both the last two integrals are positive. So Claim 4 is proved and \(u\) is a decreasing function.

\bg
{\em Claim 5}: \ \(I_u\) is an interval and \(u\) is strictly decreasing 
in \(I_u\).

\md
By Claim 4 it is immediate to deduce that \(I_u\) is an open interval. Moreover, if there exists an interval \([t_1,t_2]\subset I_u\) in which \(u\) is constant,  then  integrating equation \eqref{eq:E} we obtain 
\[\int_{t_1}^{t_2} \rho(u(t)) \ \de t =0\]
which is a contradiction since \(\rho\) is positive on \((0,1)\).

\bg
{\em Claim 6}: \ \(u'(t)<0\) for every \(t\in I_u\).

\md
Assume, by contradiction, \(u'(t_0)=0\) for some \(t\in I_u\).
Hence, by Claim 3, we have \(D(u(t_0))=0\).
 Put \(\psi(t):=D(u(t))u'(t)\). By equation \eqref{eq:E} we derive
\( \psi'(t_0)=\rho(u(t_0))<0\). On the other hand, 
\[ \frac{\psi(t)-\psi(t_0)}{t-t_0}=  \frac{\psi(t)}{t-t_0}=\frac{D(u(t))-D(u(t_0))}{u(t)-u(t_0)} \frac{ u(t)-u(t_0)}{t-t_0}u'(t)\]
so, passing to the limit as \(t\to t_0\) we infer \(\psi'(t_0)=0\), a contradiction.

\end{proof}

\bg
In view of Proposition \ref{prop:monotone u}, the following cases can occur:

\begin{itemize}
\item \ \((a,b)=\RR\); in this case we have a classical solution in \(C^1(\RR)\). The solution can be strictly decreasing in \(\RR\), hence the equilibria \(0\) and \(1\) are not reached in a finite time, or the solution can assume the value \(0\) and/or \(1\) at some \(t_0\in \RR\) and hence it is constant in a half-line.  The latter case does not occur when the Cauchy problem for equation  \eqref{eq:E} with initial conditions \(u(t_0)=u'(t_0)=0\) or \(u(t_0)=1, u'(t_0)=0\)  has a unique solution.

\item \ \((a,b)\) is a half-line; in this case we have a so-called sharp solution of type 1 or type 2. One of the equilibria is reached at a finite time and \(u\) does not admit a \(C^1-\)continuation. 

\item \ \((a,b)\) is a bounded interval; in this case we have a so-called sharp t.w.s. of type 3. Both the equilibria are reached at finite times, and \(u\) does not admit a \(C^1-\)continuation.

\end{itemize}

\bg

\medskip
\begin{center}
\includegraphics[scale=0.65]{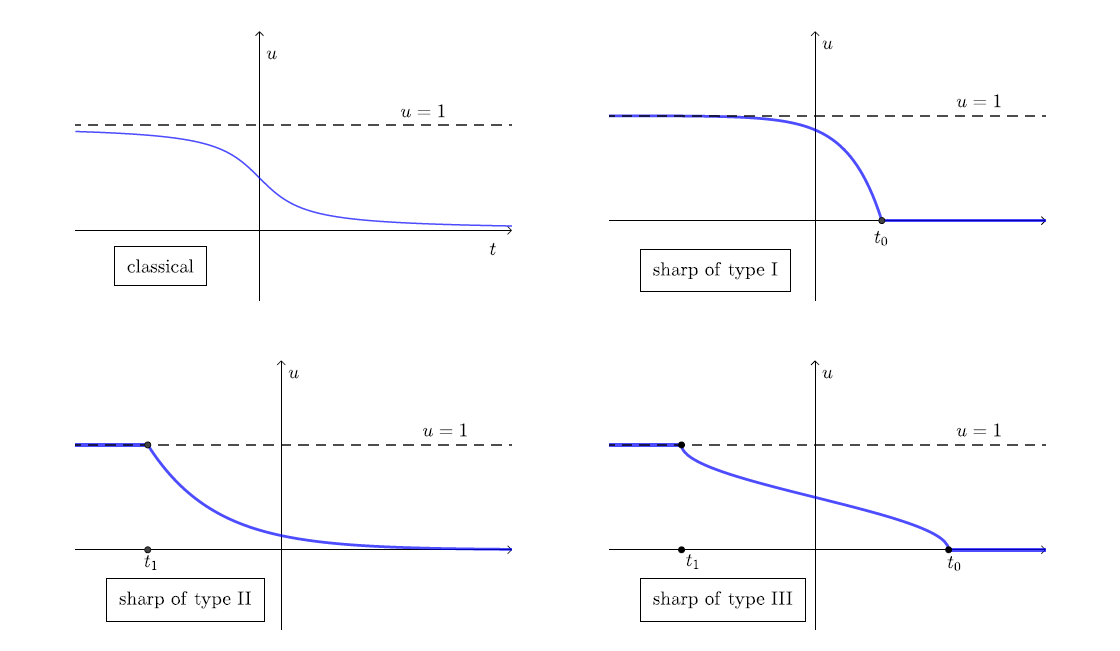}
\end{center}

\bg
We will delve deeper into this classification in the last section of the paper.

\bg
Finally, observe that if \(u\) is a t.w.s, then each traslation \(u(t-t_0)\) is a t.w.s. too; so the t.w.s. are defined up to shifts.

\section{\textcolor{black}{Reduction to a first order singular problem}}

Thanks to the monotonicity property proved in Proposition \ref{prop:monotone u} we now introduce a suitable singular boundary value problem (b.v.p.) of the first order, whose solvability is equivalent to the existence of t.w.s. for equation \eqref{eq:E}.

\md
More in detail, consider the following b.v.p. (here the dot denotes differentiation with respect to \(u\))

\beq
\begin{cases}
\overset{\cdot}{z}\left(u\right)=f(u)-cg\left(u\right)-\dfrac{\rho(u)D(u)}{z\left(u\right)}, & u\in(0,1)\setminus D_{0}\\
z\left(u\right)D\left(u\right)<0, & u\in(0,1)\setminus D_{0}\\
z(0^+)=z(1^-)=0.
%
\end{cases}\label{eq:ordine1}
\eeq

\bg

For {\em solution } of \eqref{eq:ordine1} we mean a continuous function \(z\) defined in \((0,1)\), such that \(z\in C^1((0,1)\setminus D_0)\) and \(z\) satisfies the  equalities and inequalities given in \eqref{eq:ordine1}. Of course, any solution \(z\) vanishes exactly at the points of \(D_0\).

 We now prove the equivalence between the existence of t.w.s. for equation \eqref{eq:E} and the existence of solutions of problem \eqref{eq:ordine1}.

\bg
\begin{thm} 
\label{t:equiv}
If \(u\) is a t.w.s. of equation \eqref{eq:E}, then the function \(z(u):= D(u)u'(t(u))\)  is a solution of problem \eqref{eq:ordine1}, where \(u\mapsto t(u)\) denotes the inverse function of \(u\), defined on \((0,1)\).  Moreover, the function \(z(u)/D(u)\) defined in \((0,1)\setminus D_0\) admits a continuous extension \(\phi\) defined in \((0,1)\).

Vice versa, if \(z\in C(0,1)\) is a solution of \eqref{eq:ordine1}, such that the function \( z(u)/D(u)\),  \(u\in (0,1)\setminus D_0\), admits a continuous extension \(\phi\) defined in \((0,1)\), then fixed a value \(u^*\not\in D_0\), the Cauchy problem
\beq
\begin{cases}
u'=\phi(u) \\ u(0)=u^*
\end{cases} \label{eq:prob-u}
\eeq
admits a solution \(u\), such that in its maximal existence interval \((a,b)\) it
is a t.w.s. of \eqref{eq:E}.
\end{thm}

\begin{proof}
Let \(u\) be a t.w.s. to equation \eqref{eq:E} in \((a,b)\), and put \(I_u:=\{t : \ 0<u(t)<1\}\). Let \(u\mapsto t(u)\) be the inverse function, defined in \((0,1)\) and taking valued in \(I_u\), whose existence is guaranteed by Proposition \ref{prop:monotone u}. Notice that \(u\mapsto t(u)\) actually is a \(C^1-\)function, since \(u\) is \(C^1\) with \(u'(t)\ne 0\) for every  \(t\in I_u\).
Put \(\psi(t):= D(u(t))u'(t)\), define \(z(u):= \psi(t(u))=D(u)u'(t(u))\).
Hence \(z\) is a \(C^1-\)function defined in \((0,1)\), satisfying \(z(0^+)=z(1^-)=0\) by conditions \eqref{eq:limsh}.
Moreover, for every \(u\in (0,1)\setminus D_0\) we have
\[ \dot{z} (u)= \dfrac{\psi'(t(u))}{u'(t(u))} = f(u)-cg(u)-\dfrac{\rho(u)}{u'(t(u))} = f(u)-cg(u)-\dfrac{\rho(u)D(u)}{z(u)}.\] 

Furthermore, since \(u'(t(u))<0\) for every \(u\in (0,1)\), we  
get  \(z(u)D(u)<0\) for every \(u\in (0,1)\setminus D_0\).

 Finally, since \(z(u)/D(u)=u'(t(u))\) for every \(u\in (0,1)\setminus D_0\), we have that the function \(\phi(u):=u'(t(u))\) is a continuous extension in \((0,1)\).

\md
Vice versa, let \(z\in C(0,1)\) be a solution of \eqref{eq:ordine1} such that the function \(z/D\) admits a continuous extension \(\phi\) defined  in \((0,1)\).
Then, problem \eqref{eq:prob-u} admits at least a solution \(u\), defined in its maximal existence interval \((a,b)\), with \(-\infty\le a <b\le +\infty\). 

First of all, notice that \(u\) is a decreasing function, since \(\phi(u)\le 0\) for every \(u\in (0,1)\). Moreover, we have \(u'(t)<0\) whenever \(u(t)\not \in D_0\).
So, if \(u\) is constant in some interval \([c,d]\subset (a,b)\), then \(u(t)\in D_0\) for every \(t\in [c,d]\) and it is immediate to verity that  the cut function
\[ \tilde u(t):=\begin{cases} u(t), \quad t\in (a,c]  \\ u(t-c+d), \quad t\in[c,c+b- d) \end{cases}\]
is again a \(C^1\) solution of problem \eqref{eq:prob-u}.
Since \(D_0\) is finite, there exists  at the most a finite number of intervals 
in which \(u\) is constant and so we deduce that problem \eqref{eq:prob-u} admits at least a strictly decreasing solution. 
Therefore, we can assume, without restriction, that the solution \(u\) is strictly decreasing, so that put \(T_0:=\{t : \ u(t)\in D_0\}\), we have that \(T_0\) is finite and \((a,b)\setminus T_0\) is disjoint union of a finite number of open intervals (bounded or unbounded), say \((a,b)\setminus T_0=\dis \bigcup_{j=1}^m (a_j,b_j)\).

Notice that the function  \(t\mapsto D( u(t))u'(t)\) is \(C^1\) in each interval \((a_j,b_j)\), \(j=1,\cdots,m\), with 
\[ \left(D(u(t))u'(t)\right)'=(z(u(t)))'=\dot z(u(t))u'(t)=
 (f(u(t))-c g(u(t)))u'(t) - \rho(u(t))\]
so \(u\) satisfies the differential equation in \eqref{eq:prob-u} in each interval \((a_j,b_j)\), \(j=1,\cdots,m\).
On the other hand, for every \(t_0\in T_0\) we have that 
there exists the limit

\[ \lim_{t \to t_0} \left( D(u(t))u'(t))\right)'=  (f(u(t_0))-c g(u(t_0)))u'(t_0) - \rho(u(t_0))\]
and this impplies that the function \((D\circ u)u'\) is \(C^1\) in the whole interval \((a,b)\), where it satisfies the differential equation of problem \eqref{eq:prob-u}. 

\md
The boundary conditions \(z(0^+)=z(1^-)=1\) in problem \eqref{eq:ordine1} imply the validity of conditions \eqref{eq:limsh}. 

Finally, let us prove the validity of boundary confitions \eqref{eq:as}. 
Since \(u\) is monotone and \((a,b)\) is the maximal existence interval, there exists   \(u(a^+)\le 1\) and \(u(b^-)\ge 0\). 
 If  \(u(b^-)=u_0\in (0,1)\), then necessarily \(b=+\infty\). 
 Moreover, 
\( \dis\lim_{t\to +\infty} u'(t)=\lim_{u \to u_0} \phi(u)=\phi(u_0)\), hence, necessarily \(\phi(u_0)=0\). But 
\[ \lim_{t \to +\infty} (D(u(u))u'(t))' = \lim_{t \to +\infty} \dot z(u(t))u'(t)= \lim_{u \to u_0} \dot z(u)u'(t(u))= \lim_{u\to u_0}\dot z(u_0)\phi(u_0)=0\]
and this is in contrast with equation \eqref{eq:E}, since \(\rho(u_0)\ne 0\).

%

The proof concerning the value \(u(a^+)\) is analogous.

Summarizing, \(u\) satisfies all the requirements of Definition \ref{d:tws} and is a t.w.s. of  \eqref{eq:E}.
\end{proof}

\bg

We conclude this section by discussing the uniqueness of the solutions of problem  \eqref{eq:ordine1} and of the t.w.s. (up to shifts).
To this aim, we need the following preliminary lemma, which 
 is just a remark concerning a change of variable, whose proof is immediate. We will refer to it several times later in the paper.

\bg

\begin{lem}
\label{l:negative} Let \(f,g,h\) be continuous functions defined in a given interval \((\alpha,\beta)\subset (0,1)\) 
and assume that \(h(u)<0\) in \((\alpha,\beta)\) with \(h(\alpha)=h(\beta)=0\).

Then, a function \(z\in C^1(\alpha,\beta)\) is a positive solution 
of the equation
\beq \dot z(u) =f(u)-cg(u)-\dfrac{h(u)}{z(u)}\label{eq:uconh}\eeq
if and only if 
the function \(\zeta(u):=-z(\alpha+\beta-u)\) is a negative solution of the equation
\[ \dot \zeta(u) =\tilde f(u)-c\tilde g(u)-\dfrac{ \tilde h(u)}{\zeta(u)}\]
where
\[ \tilde f(u):=f(\alpha+\beta-u), \ \tilde g(u):=g(\alpha+\beta-u), \ \tilde h(u):=-h(\alpha+\beta-u)\] 
with \(\tilde h(u)>0\) in \((\alpha,\beta)\) and \(\tilde h(\alpha)=\tilde h(\beta)=0\).

\end{lem}
\bg

\begin{prop}
\label{p:unique}
For every fixed \(c\in \RR\) problem \eqref{eq:ordine1} can admit, at the most, one solution.
Moreover, foe every fixed \(c\in \RR\) equation \eqref{eq:E} can admit, at the most, one t.w.s. (up to shifts).
\end{prop}

\begin{proof}
Assume, by contradition, that for some \(c\in \RR\) there exists   a pair \(z_1, z_2\) of different solutions of problem \eqref{eq:ordine1}. Since \(z_1(u_0)=z_2(u_0)\) for every \(u_0\in D_0\), there exists an interval \((\alpha,\beta)\subset (0,1)\), with \(\alpha,\beta\in D_0\cup \{0,1\}\) such that \(D\) does not change sign in \((\alpha,\beta)\) and both the functions \(z_1,z_2\) are solution of the  problem

\[
\begin{cases}
\overset{\cdot}{z}\left(u\right)=f(u)-cg\left(u\right)-\dfrac{\rho(u)D(u)}{z\left(u\right)}, & u\in(\alpha,\beta)\\
z\left(u\right)D\left(u\right)<0, & u\in(\alpha,\beta)\\
z(\alpha^-)=z(\beta^+)=0
\end{cases}
\]

If \(D(u)>0\) in \((\alpha,\beta)\) this in contradiction with the uniqueness result \cite[Proposition 13]{CaMaPa}. The proof  in the case  \(D(u)<0\) in \((\alpha,\beta)\) follows from Lemma \ref{l:negative}.

\bg
Let us now assume that equation \eqref{eq:E} admits two different t.w.s. \(u_1\) defined in \((a_1,b_1)\) and \(u_2\) defined in \((a_2,b_2)\). Let \(t_1:(0,1)\to (a_1,b_1)\) and \(t_2:(0,1)\to (a_2,b_2)\) be the inverse functions of \(u_1\), \(u_2\) respectively. Since \(u_i(t_i(u))\equiv u\), \(i=1,2\), we get 
\beq
\label{eq:inverse}
u_1'(t_1(u))\cdot \dot{t_1}(u)=1 \ , \quad u_2'(t_2(u))\cdot \dot{t_2}(u)=1
 \eeq
 By Theorem \eqref{t:equiv} the functions \(z_1(u):=D(u)u_1'(t_1(u))\), \(z_2(u):=D(u)u_2'(t_2(u))\), are both solutions of problem \eqref{eq:ordine1}. So, by the uniqueness just proved, we have \(z_1(u)=z_2(u)\). Hence, for every \(u \in (0,1)\setminus D_0\) we have \(u_1'(t_1(u))=u_2'(t_2(u))\).
Therefore, by \eqref{eq:inverse} we deduce \(\dot{t_1}(u)=\dot{t_2}(u)\) for every \(u\in (0,1)\). Then, since \(D_0\) is finite, by the continuity of the inverse functions \(t_1, t_2\) we infer the exitence of a constant  \(k\in \RR\) such that 
\(t_2(u)=t_1(u)+k\) for every \(u\in (0,1)\). So, \(u_2(t)=u_1(t-k)\), that is \(u_2\) is a shift of \(u_1\).

\end{proof}

\section{Solvability of the first order singular problem}

\bg
In view of Theorem \ref{t:equiv},  we now investigate the solvability of problem \eqref{eq:ordine1}. For the sake of simplicity, along this section we consider 
a generic function \(h\in C([0,1])\), instead of the product \(\rho(u)D(u)\), such that 
put  
\[ H_0:=\{u \in [0,1]: \ h(u)=0\}\]
we have 
\beq
\label{ip:segnoh}
 \ H_0   \text{ is finite \ and \  } 0,1\in H_0.
\eeq

%

\md

We study the solvability of the following 
 boundary value problem
\beq 
\label{eq:pr-h} 
\begin{cases}
\overset{\cdot}{z}\left(u\right)=f(u)-cg\left(u\right)-\dfrac{h(u)}{z\left(u\right)}, & u\in(0, 1)\setminus H_0\\
z\left(u\right)h(u)<0, &  u\in(0, 1)\setminus H_0 \\
z(0^+)=z(1^-)=0.
%
\end{cases}
\eeq

\md

\md

The existence and non-existence of solutions of problem \eqref{eq:pr-h} has been investigated in \cite{CaMaPa} in the case \(h\) is a positive function in \((0,1)\).  

\bg
The following results summarizes some statements of Theorem 12, 14 and 15 in \cite{CaMaPa}.

\bg
\begin{thm}
\label{t:pos}
Let \(f,g,h\) be continuous functions defined in \([\alpha,\beta]\) such that \(h\) is positive on \((\alpha,\beta)\) with \(h(\alpha)=h(\beta)=0\) and differentiable at \(\alpha,\beta\).  

Suppose that 

\[
g(\alpha)>0 \quad \textit{ and } \quad \int_\alpha^u g(s) \ \de s >0 \ \textit{ for every } u\in (\alpha,\beta].
\]

Then, there exists a value \(c^*\) such that 
problem 
\[ \begin{cases}
\overset{\cdot}{z}\left(u\right)=f(u)-cg\left(u\right)-\dfrac{h(u)}{z\left(u\right)}, & u\in(\alpha, \beta)\\
z\left(u\right)<0, &  u\in(\alpha, \beta)\\
z(\alpha^+)=z(\beta^-)=0
%
\end{cases}
\]
 admits solution if and only if \(c\ge c^*\).

Moreover, for every \(c\ge c^*\) the solution \(z_c\)  is differentiable at  \(\alpha,\beta\), with
\beq
\label{eq:puntoz}
\dot z(\alpha)= \begin{cases} r_+(c,\alpha) & \text{ if } c>c^* \\ r_-(c,\alpha) & \text{ if } c=c^* \end{cases}; \qquad \dot z(\beta)=r_+(c,\beta)
\eeq
where
\beq r_\pm (c,u):=\frac12 \left(f(u)-cg(u)\pm \sqrt{(f(u)-cg(u))^2-4\dot h(u)} \right), \quad u\in \{\alpha,\beta\}.
\label{d:rpm}\eeq

Finally,  put 
\[ m_g:=\inf_{u\in (\alpha,\beta)} -\!\!\!\!\!\!\int_\alpha^u g(s)  \de s ; \quad M_f:=\sup_{u\in (\alpha,\beta)}  -\!\!\!\!\!\!\int_\alpha^u f(s)  \de s \quad M_h:= \sup_{u\in (\alpha,\beta)}  -\!\!\!\!\!\!\int_\alpha^u \frac{h(s)}{s-\alpha} \ \de s \]
we have that \(m_g>0\); \(M_f, M_h<+\infty\) and  \(c^*\) satisfies
\beq
\label{eq:cstar}
\frac{2\sqrt{\dot h(\alpha)}+f(\alpha)}{g(\alpha)}\le c^*\le \frac{2\sqrt{ M_h}+M_f}{m_g}.
\eeq

\end{thm}

\begin{proof} First of all, note that   \(m_g>0\) and \(M_f, M_h\) are finite.

Indeed, put \(G(u):= -\!\!\!\!\!\!\dis\int_\alpha^u g(s)  \de s\),  we have that \(G\) is a positive continuous function in \((\alpha,\beta]\) such that \(\dis\lim_{u\to \alpha^+} G(u)=g(\alpha)>0\). 
So, \(m_g>0\). Similarly,  put \(F(u):= -\!\!\!\!\!\!\dis\int_\alpha^u f(s)  \de s\),  we have that \(F\) is a  continuous function in \((\alpha,\beta]\) such that \(\dis\lim_{u\to \alpha^+} F(u)=f(\alpha)\). 
So, \(M_f\) is finite. Finally, also \(H(u):= -\!\!\!\!\!\!\dis\int_\alpha^u \frac{h(s)}{s-\alpha} \ \de s\),   is a positive  continuous function in \((\alpha,\beta]\) with  \(\dis\lim_{u\to \alpha^+} H(u)=\dot h(\alpha)\). Hence, \(M_h\) is finite too.

\md
Let us fix a positive value \(c > \dfrac{2\sqrt{M_h}+M_f}{m_g}\).  
Then, 
\[ \begin{array}{ll} \dis \inf_{u \in (\alpha,\beta) }  -\!\!\!\!\!\!\dis\int_\alpha^u (cg(s)-f(s))  \de s & \ge c  \dis\inf_{u \in (\alpha,\beta) }  -\!\!\!\!\!\!\dis\int_\alpha^u g(s) \de s + \dis \inf_{u \in (\alpha,\beta) }  -\!\!\!\!\!\!\dis\int_\alpha^u -f(s) \de s  \\ & = c  m_g - \dis\sup_{u \in (\alpha,\beta) }  -\!\!\!\!\!\!\dis\int_\alpha^u f(s) \de s \\
 & = cm_g-M_f  > 2 \sqrt{M_h} = 2
 \sqrt{\dis\sup_{u\in (\alpha,\beta)} -\!\!\!\!\!\!\dis\int_\alpha^u \dis\frac{h(s)}{s-\alpha} \ \de s\ }. \end{array}\]

Therefore, assumption (26) in \cite[Theorem 12]{CaMaPa} is satisfied and problem \eqref{eq:ordine1} admits a solution. Morever, by \cite[Theorem 14]{CaMaPa} we deduce the existence of a value \(c^*\) such that problem \eqref{eq:ordine1} admits a solution if and only if \(c\ge c^*\).

Since \(c^*\) is the minimal wave speed, by what we have just proved, necessarily \(c^*\le  \dfrac{2\sqrt{M_h}+M_f}{m_g}\). Finally,  the inequality \(c^*\ge \dfrac{2\sqrt{\dot h(\alpha)}+f(\alpha)}{g(\alpha)}\) follows from assertion (27) in \cite[Theorem 14]{CaMaPa}.

\end{proof}

\bg

By using the change of variable given in Lemma \ref{l:negative} we obtain the following result in the case \(h\) is negative in \((\alpha, \beta)\).

\bg
\begin{thm}
\label{t:neg}
Let \(f,g,h\) be continuous functions defined in \([\alpha,\beta]\) such that \(h\) is negative on \((\alpha,\beta)\) with \(h(\alpha)=h(\beta)=0\) and differentiable at \(\alpha,\beta\).  

Suppose that 

\[
g(\beta)>0 \quad \textit{ and } \quad \int_u^\beta g(s) \ \de s >0 \ \textit{ for every } u\in [\alpha,\beta).
\]

Then, there exists a value \(c^*\) such that 
such that problem 
\beq 
\label{eq:ordine1neg} \begin{cases}
\overset{\cdot}{z}\left(u\right)=f(u)-cg\left(u\right)-\dfrac{h(u)}{z\left(u\right)}, & u\in(\alpha, \beta)\\
z\left(u\right)>0, &  u\in(\alpha, \beta)\\
z(\alpha^+)=z(\beta^-)=0
%
\end{cases}
\eeq
 admits solution if and only if \(c\ge c^*\).

Moreover, 
 for every \(c\ge c^*\) the solution \(z_c\)  is differentiable at  \(\alpha,\beta\), with {\rm (}see \eqref{d:rpm}{\rm )}
\beq
\label{eq:puntozneg}
\dot z(\alpha)= r_+(c,\alpha), \qquad \dot z(\beta) = \begin{cases} r_+(c,\beta) & \text{ if } c>c^* \\ r_-(c,\beta) & \text{ if } c=c^*. \end{cases}
\eeq

Finally,
put 
\[  m_g^*:=\inf_{u\in (\alpha,\beta)} -\!\!\!\!\!\!\int_u^\beta g(s)  \de s ; \quad  M_f^*:=\sup_{u\in (\alpha,\beta)}  -\!\!\!\!\!\!\int_u^\beta f(s)  \de s \quad  M_h^*:= \sup_{u\in (\alpha,\beta)}  -\!\!\!\!\!\!\int_u^\beta \frac{h(s)}{s-\beta} \ \de s \]
we have that \( m_g^*>0\); \(M_f^*,  M_h^*<+\infty\) and  \(c^*\) satisfies
\beq
\label{eq:cstarneg}
\frac{2\sqrt{\dot h(\beta)}+f(\beta)}{g(\beta)}\le c^*\le \frac{2\sqrt{ M_h^*}+ M_f^*}{ m_g^*}.
\eeq

\begin{proof}
The proof follows from Theorem \ref{t:pos} and Lemma \ref{l:negative}. Indeed, by Lemma \ref{l:negative} we get that \(z\) is a solution of problem \eqref{eq:ordine1neg}  for some \(c\) if and only if the function \(\zeta(u):= -z(\alpha+\beta-u)\) is a solution of the associated problem 
\[
\begin{cases}
\overset{\cdot}{\zeta}\left(u\right)=\tilde f(u)-c\tilde g\left(u\right)-\dfrac{\tilde h(u)}{\zeta\left(u\right)}, & u\in(\alpha, \beta)\\
\zeta\left(u\right)<0, &  u\in(\alpha, \beta)\\
\zeta(\alpha^+)=\zeta(\beta^-)=0.
\end{cases}
\]
Notice that \(\dot \zeta(\alpha)=\dot z(\beta)\) and \(\dot \zeta(\beta)=\dot z(\alpha)\), so \eqref{eq:puntozneg} holds.

Moreover, observe that 
 if \(\gamma \) is a continuous function in \((\alpha,\beta)\), then put \(\tilde \gamma(u):=\gamma(\alpha+\beta-u)\) we have
\[   -\!\!\!\!\!\!\int_u^\beta \gamma(s) \ \de s =   -\!\!\!\!\!\!\int_\alpha^{\alpha+\beta-u} \tilde \gamma(t) \ \de t\]
so
\[ \inf_{u\in (\alpha,\beta)}  -\!\!\!\!\!\!\int_u^\beta \gamma(s) \ \de s = \inf_{v\in (\alpha,\beta)}  -\!\!\!\!\!\!\int_\alpha^v \tilde \gamma(t)\ \de t\]
and the same argument holds when taking the supremum of the mean value. Finally,
\[\begin{split}\sup_{u\in (\alpha,\beta)} -\!\!\!\!\!\!\int_u^\beta \frac{h(s)}{s-\beta} \de s =\sup_{u\in (\alpha,\beta)}
\frac{\dis\int_u^{\beta} \frac{-\tilde h(\alpha+\beta-s)}{s-\beta}  \de s}{\beta-u}    \\ =
\sup_{u\in (\alpha,\beta)}
\frac{\dis\int_\alpha^{\alpha+\beta-u} \frac{\tilde h(t)}{t-\alpha} \de t}{\beta-u}= \sup_{v\in (\alpha,\beta)} -\!\!\!\!\!\!\int_\alpha^v \frac{\tilde h(t)}{t-\alpha} \de t.\end{split}\]
So, estimate \eqref{eq:cstarneg} is the same as  \eqref{eq:cstar} for the problem \eqref{eq:ordine1neg}.
\end{proof}
\end{thm}

\bg

We now consider problem \eqref{eq:pr-h} with a changing-sign function \(h\).

\md
Of course, the open set \((0,1)\setminus H_0\) is union of a finite number of disjont intervals  (see \eqref{ip:segnoh})
\beq\label{eq:hsplit}
(0,1)\setminus H_0= \bigcup_{k=1}^n (\alpha_k,\beta_k).
\eeq
In what follows it will be convenient to  distinguish the indexes \(k\) such that \(h\) is positive in \((\alpha_k,\beta_k)\) from those such that \(h\) is negative  in \((\alpha_k,\beta_k)\). Hence, we set
\beq \label{d:kpm} \begin{split} K^+:=\{k\in \{1,\cdots,n\}: \ h(u)>0 \text{ in  } (\alpha_k,\beta_k)\};\\  K^-:=\{k\in \{1,\cdots,n\}: \ h(u)<0 \text{ in  } (\alpha_k,\beta_k)\}.\end{split} \eeq

%
%
%

Finally, let us define the following constants:
\beq
\label{d:notcstar}
\begin{array}{ll}
G_k^+:= \dis\inf_{u\in (\alpha_k,\beta_k)} -\!\!\!\!\!\!\int_{\alpha_k}^u g(s)  \de s, \quad & G_k^-:=\dis\inf_{u\in (\alpha_k,\beta_k)} -\!\!\!\!\!\!\int_u^{\beta_k} g(s)  \de s \\ F_k^+:= \dis\sup_{u\in (\alpha_k,\beta_k)} -\!\!\!\!\!\!\int_{\alpha_k}^u f(s)  \de s, \quad & F_k^-:=\dis\sup_{u\in (\alpha_k,\beta_k)} -\!\!\!\!\!\!\int_u^{\beta_k} f(s)  \de s \\
H_k^+:= \dis\sup_{u\in (\alpha_k,\beta_k)} -\!\!\!\!\!\!\int_{\alpha_k}^u \frac{h(s)}{s-\alpha_k} \ \de s, \quad & H_k^-:=\dis\sup_{u\in (\alpha_k,\beta_k)} -\!\!\!\!\!\!\int_u^{\beta_k} \frac{h(s)}{s-\beta_k} \  \de s 

\end{array}
\eeq

\md
We can now  prove the following existence result for the general problem \eqref{eq:pr-h}.

\bg
\begin{thm}\label{t:gen}
Let \(f,g,h\) be continuous functions defined in \([0,1]\) such that condition \eqref{ip:segnoh} holds true. Assume that    \(h\) is differentiable at each point \( u_0\in H_0 \cap (0,1).\)
Moreover, suppose that {\rm (}see \eqref{d:kpm}{\rm )}:

\beq\label{ip:g1}
g(\alpha_k)>0 
\text{ and } \int_{\alpha_k}^u g(s)  \de s >0,\ \text{ for every } u\in (\alpha_k,\beta_k], \ \text{ whenever } k \in K^+;  
\eeq

\beq\label{ip:g2}
g(\beta_k)>0 
\text{ and } \int_u^{\beta_k} g(s)  \de s >0,\ \text{ for every } u\in  [\alpha_k,\beta_k), \ \ \text{ whenever } k \in K^-. 
\eeq


Then, there exists a value \(c^*\in \RR\)
such that problem \eqref{eq:pr-h}
admits a solution \(z_c\) if and only if \(c \ge c^*\) and it is not solvable for every \(c<c^*\).

\md
Moreover, with the notations given in  \eqref{d:kpm} and \eqref{d:notcstar}, we have
\beq
\label{eq:stimagen} 
\begin{split}
\max\left\{\max_{k\in K^+} \frac{2\sqrt{\dot h(\alpha_k)}+f(\alpha_k)}{g(\alpha_k)}  \ , \ \max_{k\in K^-}\frac{2\sqrt{\dot h(\beta_k)}+f(\beta_k)}{g(\beta_k)} \right\} \le \\ \le c^* \le \max\left\{\max_{k\in K^+} \frac{2\sqrt{H_k^+}+F_k^+}{G_k^+}  \ , \ \max_{k\in K^-} \frac{2\sqrt{H_k^-}+F_k^-}{G_k^-}  \right\},
\end{split}
\eeq
 (where the maxima involved in  \eqref{eq:stimagen} have to be intended as \(-\infty\) if the set to which they refer is empty).

Finally, for every \(c>c^*\) the solution \(z_c\) is differentiable at every point \(u_0\in H_0\) with 
\beq\label{eq:zpuntoh}
\dot z(u_0)= \frac12 \left(f(u_0)-cg(u_0)+ \sqrt{(f(u_0)-cg(u_0))^2-4\dot h(u_0)} \right).
\eeq

\end{thm}

\begin{proof}
In each interval \((\alpha_k,\beta_k)\) we can apply 
  Theorem \ref{t:pos} or Theorem \ref{t:neg} according to the sign of \(h\). Hence, we can deduce  that  for every \(k=1,\cdots,n\)  there exists a value \(c_k^*\) such that for every \(c\ge c_k^*\) there exists a function \(z_{k,c}\in C^1(\alpha_k,\beta_k)\), 
satisfying the differential equation of \eqref{eq:pr-h} in the same interval and such that 
\(z_{k,c}(u)h(u)<0\) for every \(u\in (\alpha_k,\beta_k)\), with \(z_{k,c}(\alpha_k^+)=z_{k,c}(\beta_k^-)=0\).

\sm
Put \(c^*:= \max\{c_1^*, \cdots,c_n^*\}\). For every  \(c\ge c^*\) let us 
define \(z_c:[0,1]\to \RR\) by setting
\[ z_c(u)=\begin{cases} z_{k,c}(u) & \text{ if } u\in (\alpha_k,\beta_k), \ k=1,\cdots,n \\
0 & \text{ elsewhere. } \end{cases}\]
Of course, \(z_c\) is continuous in \([0,1]\), continuously differentiable on \((0,1)\setminus H_0\), and it satisfies the differential equation of problem \eqref{eq:pr-h} in \((0,1)\setminus H_0\). Moreover \(z_c(u)h(u)<0\) for every \(u\in (0,1)\setminus H_0\) and \(z_c(0^+)=z_c(1^-)=0\). So, \(z_c\) is a solution of problem \eqref{eq:pr-h}.

\sm
Instead, if we fix \(c<c^*\), then \(c<c_k^*\) for some \(k\in \{1,\cdots,n\}\), hence there is no function \(z\in C^1((\alpha_k,\beta_k))\) satisfying the differential equation of \eqref{eq:pr-h}, such that \(z(u)h(u)<0\) in \((\alpha_k,\beta_k)\) and  vanishing at \(\alpha_k,\beta_k\). This implies that problem \eqref{eq:pr-h} does not admit solutions.
 
 Moreover, since \(c^*=c_k^*\) for some \(k\in \{1,\cdots,n\}\), estimate \eqref{eq:stimagen} follows from \eqref{eq:cstar} and \eqref{eq:cstarneg}.

Finally, the validity of \eqref{eq:zpuntoh} for \(c>c^*\) is a consequence of \eqref{eq:puntoz} and \eqref{eq:puntozneg}.
\end{proof}

\bg
\begin{rem}\label{r:cstar}
\rm
Taking into account \eqref{eq:puntoz} and \eqref{eq:puntozneg}, when \(c=c^*\) we can not state anything about the differentiability of the solution \(z_{c^*}\) at the points of \(H_0\). Indeed,  
if \(c^*=c_j^*\) for some  \(j\in K^+\), with  \(\alpha_j\in (0,1)\), and \(\alpha_j=\beta_{j-1}\) with \(j-1\in K^-\), then  by \eqref{eq:puntoz} and \eqref{eq:puntozneg} we obtain

\[ \dot z_{c^*}(\alpha_j^+)=\dot z_{j,c_j^*}(\alpha_j)= r_-(c_j^*,\alpha_j)\]

\[  \dot z_{c^*}(\beta_{j-1}^-)=\dot z_{j-1,c_j^*}(\beta_{j-1})=\begin{cases} r_-(c_j^*,\alpha_j) & \text{ if } c_j^*=c_{j-1}^* \\ r_+(c_j^*,\alpha_j) & \text{ if } c_j^*>c_{j-1}^*. \end{cases}\]

So, if \(c_j^*>c_{j-1}^*\) the gluing \(z_{c^*}\) is not differentiable at \(\alpha_j\). Recalling that the threshold values \(c_j^*,c_{j-1}^*\) are unknown in general, it is unknown also the existence of a solution for \(c=c^*\). Hovewer, if such a solution \(z_{c^*}\) exists for \(c=c^*\) then it satisfy  (see \eqref{eq:puntoz} and \eqref{eq:puntozneg})
\[ \dot z_{c^*}(\alpha_j)= \frac12 \left( f(\alpha_j)-c^*g(\alpha_j)-\sqrt{(f(\alpha_j-c^*g(\alpha_j))^2 - 4 \dot h(\alpha_j))} \right). \]

\sm

On the other hand, if \(K^+\subset \{1\}\), that is if \(h(u)<0 \) for every \(u\in (0,1)\) or  \(h(u)(u-u_0)<0\) for  every  \(u\ne u_0\), for some \(u_0\in (0,1)\), then  this situation does not occur  and we can state that there exists a solution  also for \(c=c^*\).  We will treat this case in Corollary \ref{c:one}.

\end{rem}

\section{Existence of t.w.s.}\label{s:exis}
\md

In this Section we combine the results of the previous sections in order to derive the existence and non-existence results for the t.w.s. of equation \eqref{eq:E}.
 
In whats follows we will adopt the same notation introduced in the previous section in \eqref{eq:hsplit}, \eqref{d:kpm}  and \eqref{d:notcstar}, where in this case the function \(h\) is replaced by \(D(u)\rho(u)\),  so 
\beq \label{d:kpmD}
D_0=H_0\setminus\{0,1\} \quad \text{ and } \quad (0,1)\setminus D_0= \bigcup_{k=1}^n (\alpha_k,\beta_k). 
\eeq

Moreover, we put 
\[ D_{00}:=\{u\in (0,1): \ D(u)=\dot D(u)=0\}.\]
Of course,  \(D_{00}\) can be empty.

\md
The main result of this section is the following.
\bg

\begin{thm}\label{t:main2}
Suppose that  \eqref{ip:g1} and \eqref{ip:g2} are satisfied.
Moreover, assume that 
\beq\label{ip:nonzero}
g(u)> 0 \ \text{ for every } \ \ u\in D_{00}.
\eeq

Then, there exists a value \(\hat c \) 
such that equation \eqref{eq:E} admits t.w.s. for every  \(c>\hat c \), whereas no t.w.s. exists for \(c<\hat c \).  Finally, put 
\[ K_0^-:=\{ k=2,\cdots,n \ : \ k \in K^- \text{ and } \dot D(\alpha_k)=0\} \]
 and considered the quantities defined in \eqref{d:notcstar}, 
we have that \(\hat c \) satisfies the estimate 
\beq
\label{eq:stimagen3} 
\begin{split}
\max\left\{\max_{k\in K^+} \frac{2\sqrt{\dot h(\alpha_k)}+f(\alpha_k)}{g(\alpha_k)}  \ , \ \max_{k\in K^-}\frac{2\sqrt{\dot h(\beta_k)}+f(\beta_k)}{g(\beta_k)}\ , \ \max_{k\in K_0^-}\frac{f(\alpha_k)}{g(\alpha_k)} \right\} \le \\ \le \hat c  \le \max\left\{\max_{k\in K^+} \frac{2\sqrt{H_k^+}+F_k^+}{G_k^+}  \ , \ \max_{k\in K^-} \frac{2\sqrt{H_k^-}+F_k^-}{G_k^-}, \ \max_{k\in K_0^-}\frac{f(\alpha_k)}{g(\alpha_k)}   \right\}.
\end{split}
\eeq
 (where the maxima involved in  \eqref{eq:stimagen3} have to be intended as \(-\infty\) if the set to which they refer is empty).

\end{thm}

\md
The proof of the previous theorem needs some auxiliary results. Indeed, as it is clear by Theorems \eqref{t:equiv} and \eqref{t:gen}, the existence of t.w.s. depends on the solvability of problem \eqref{eq:pr-h} by a solution \(z\) such that the ratio \(z/D\) has a continuous extension in \((0,1)\). This is trivial when \(\dot D(u)\ne 0\) for every \(u\in D_0\), whereas it requires  an accurate study when \(D_{00}\) is nonempty. 

\md
Let us first recall the main tool of the upper and lower-solutiuon technique, which is a simple consequence of Gronwall's Lemma.  Recall that a lower-solution [resp. upper-solution] of equation \eqref{eq:uconh} is a function \(\xi\in C^1(\alpha,\beta)\) such that  
\[ \dot \xi (u)\ \le [\ge] \  f(u)-cg(u)-\frac{h(u)}{\xi(u)}\quad \text{ for every } u\in (\alpha,\beta).\]
When the above inequality is strict, then the function \(\xi\) is said to be a strict lower-solution [upper-solution]. 

\md 
\begin{lem}\label{l:Gronwall}
Let \(z\), \(\zeta\) respectively be a solution and an upper-solution of equation \eqref{eq:uconh} in an interval \(I\subset (\alpha,\beta)\)  and let \(u_0\in I\) be fixed.
Then,
\begin{itemize}
\item \ if \(z(u_0)\le \zeta(u_0)\), then  \(z(u)\le \zeta(u)\) for every \(u\ge u_0\)
\item \ if \(z(u_0)\ge \zeta(u_0)\), then  \(z(u)\ge \zeta(u)\) for every \(u\le u_0\).
\end{itemize}
Instead, if \(\zeta\) is a lower-solution, then
\begin{itemize}
\item \ if \(z(u_0)\ge \zeta(u_0)\), then  \(z(u)\ge \zeta(u)\) for every \(u\ge u_0\)
\item \ if \(z(u_0)\le \zeta(u_0)\), then  \(z(u)\le \zeta(u)\) for every \(u\le u_0\).
\end{itemize}

\end{lem}

\bg
The next result regards the behaviour of \(z/D\) near the points of \(D_{00}\). The proof follows  arguments developed in similar contexts in \cite[Theorem 2.5]{CM1} and \cite[Lemma 9.1]{becoma}.

\bg
\begin{prop}\label{p:limite}
Let all the assumptions of Theorem \ref{t:gen} be satisfied, for \(h(u):=D(u)\rho(u)\). Let \(z\) be a solution of problem \eqref{eq:pr-h} for some  \(c\ge c^*\) and let \(u_0\in D_{00}\).  Let us denote by \(\dot z_+(u_0)\) and \(\dot z_-(u_0)\) the right and the left derivative of \(z\) at \(u_0\).

 Then, if \( \dot z_+(u_0)=0\) {\rm [}resp. \( \dot z_-(u_0)=0\){\rm ]}
we necessarily have \(f(u_0)-cg(u_0)\le 0\) and moreover
\[ \lim_{u\to u_0^+} \ [ \lim_{u\to u_0^-}] \ \frac{D(u)}{z(u)} = \begin{cases} 0 & \text{ if } f(u_0)-cg(u_0)=0 \\ 
\dfrac{f(u_0)-cg(u_0)}{\rho(u_0)} & \text{ if } f(u_0)-cg(u_0)<0.\end{cases}\]
\end{prop}

\md 
\begin{proof}  We limit ourselves to give the proof for the limit as \(u\to u_0^+\), since the limit as \( u \to u_0^-\) can be treated by using the change of variable \( \zeta(u):=-z(\alpha+\beta-u)\) considered in Lemma \ref{l:negative}, for \((\alpha,\beta)=(0,1)\). Indeed, in this case the limit  \(\dis\lim_{u\to u_0^-} \frac{D(u)}{z(u)}\)  is replaced by the limit \(\dis\lim_{u\to v_0^+} \frac{\tilde D(u)}{\zeta(u)}\), where \(v_0:=1-u_0\), \(\tilde D(u)= -D(1-u)\) and \(\zeta(u)=-z(1-u)\).
 
\md
Let \(u_0\in D_{00}\) be fixed and assume that \(\dot z_+(u_0)=0\). Let \((u_n)_n\) be a decreasing sequence converging to \(u_0\) such that \(\dot z(u_n)\to 0\). 

Notice that by the differential equation in \eqref{eq:pr-h} we have
\beq \label{eq:frac zd}
 \frac{f(u)-cg(u)-\dot z(u)}{\rho(u)} = \frac{D(u)}{z(u)} <0 \quad \text{for every } u\in (0,1)\setminus D_0.
\eeq

 Therefore, since  \(\dot z(u_n)\to 0\), we deduce that 
\(f(u_0)-cg(u_0)\le 0\).

\md 
{\bf Case 1:}  \(f(u_0)-cg(u_0)=0\). 
\md

 Fixed a value \(\epsilon>0\), put 
\[
\omega_{\varepsilon}\left(u\right):=-\frac{D\left(u\right)}{\varepsilon}.
\]
Since \(\dot \omega_\epsilon (u_0)=0\) and \(\dis\lim_{u\to u_0^+} f(u)-cg(u)-\frac{h(u)}{\omega_\epsilon (u)}=\epsilon \rho(u_0)>0,\)
for some positive \(\delta^*=\delta^*(\varepsilon)\) we have
\[ \dot \omega_\epsilon (u) <  f(u)-cg(u)-\frac{h(u)}{\omega_\epsilon (u)} \quad \text{ for every } u\in (u_0, u_0+\delta^*) \]
that is \(\omega_\varepsilon\) is a strict lower-solution for equation in \eqref{eq:pr-h}.

 From \eqref{eq:frac zd} we deduce that
\(  D(u_n)/z(u_n)\to 0\). Therefore, for every  \(n\) sufficiently large we have 
\(D(u_{n})/z(u_{ n}) > -\varepsilon \), that is 
\beq \label{eq:omega1}
 \dfrac{\omega_\varepsilon (u_{n})}{z(u_{n})}<1 \quad \text{ for every } n \text{ large enough}. \eeq
Then, if \(D\) is positive in a right neighborhood of \(u_0\), then \(z\) is negative in the same neighborhood and we get 
\(z(\bar u)<\omega_\epsilon (\bar u)\) for some \(\bar u>u_0\) and recalling that \(\omega_\varepsilon\) is a lower-solution  we conclude that 
\(z(u)<\omega_\varepsilon(u)\) for every \(u\in (0, \bar u)\), that is 
\[ -\varepsilon < \frac{D(u)}{z(u)}<0 \quad \text{ for every } u\in (u_0, \bar u).\]

\md

Instead, if 
\(D\) is negative in a right neighborhood of \(u_0\), then \(z\) is positive in the same neighborhood and by \eqref{eq:omega1} we get the existence of a natural \(\bar n\) such that 
\[ z(u_n)> \omega_\epsilon  (u_n) \quad \text{ for every } n\ge \bar n.\]
This implies that for every \(n\ge \bar n\) we have 
 \(z(u)\ge  \omega_\epsilon  (u)\)  for every \(u\in (u_{n+1},u_n)\). Indeed, if \(z(\bar u)< \omega_\epsilon  (\bar u)\) for some \(\bar u\in (u_{n+1},u_n)\) and \(n\ge \bar n\), then there exists a value \(u^*\in (u_{n+1},\bar u)\) such that 
 \(z(u^*)= \omega_\varepsilon (u^*)\) and \(\omega_\varepsilon (u)> z(u)\) in a right neighborhood of \(u^*\), in contradiction with Lemma \ref{l:Gronwall}.
 
 \md
 Summarizing, whatever it is the sign of \(D\) in a right neighborhood of \(u_0\) we have 
 \[ \lim_{u \to u_0^+} \frac{D(u)}{z(u)}=0.\]
 
 \sm

\bg 
{\bf Case 2:} $f\left(u_{0}\right)-cg\left(u_{0}\right)<0$.

\md

 Taking account of \eqref{eq:frac zd}  we get 
\beq\label{eq:limun}
\frac{D(u_n)}{z(u_n)} \to \frac{ f(u_0)-cg(u_0)}{\rho(u_0)}<0 \quad \text{ as  } n\to +\infty.
\eeq

Fixed a positive \(\varepsilon< (cg(u_0)-f(u_0))/\rho(u_0)\),  let us define  
\[
\eta_{\varepsilon}\left(u\right):=-\frac{\rho\left(u_{0}\right)}{cg\left(u_{0}\right)-f\left(u_{0}\right)-\varepsilon\rho\left(u_{0}\right)}D\left(u\right).
\]

Since \(\dot \eta_\varepsilon (u_0)=0\) and 
\(\dis \lim_{u \to u_0^+} f(u)-cg(u)-\frac{h(u)}{\eta_\varepsilon(u)}= -\varepsilon \rho(u_0)<0\), we get that 
$\eta_{\varepsilon}$ is a strict upper-solution for the equation in \eqref{eq:pr-h} in \((u_0, u_0+\delta)\) for some \(\delta>0\).

On the other hand, from \eqref{eq:limun} we derive 
\beq\label{eq:limun2}
\frac{\eta_\eps(u_n)}{z(u_n)} \to \frac{cg\left(u_{0}\right)-f\left(u_{0}\right)}{cg\left(u_{0}\right)-f\left(u_{0}\right)-\varepsilon\rho\left(u_{0}\right)}>1 \quad \text{ as  } n\to +\infty.
\eeq

\sm
Therefore, if \(D\) is positive in a right neighborhood of \(u_0\) (hence \(z\) is negative) we infer that 
\(\eta_\eps(\tilde u)< z(\tilde u)\) for some \(\tilde u\in (u_0, u_0+\delta)\) and by applying Lemma \ref{l:Gronwall} we achieve that \(\eta_\eps(u)<z(u)\) in \((u_0,\tilde u)\), that is 
\[ \frac{D(u)}{z(u)} < \frac{ f(u_0)-cg(u_0)}{\rho(u_0)}+\eps \quad \text{ for every } u\in (u_0,\tilde u).  \]

Instead, if \(D\) is negative in a right neighborhood of \(u_0\) (hence \(z\) is positive), by \eqref{eq:limun2} we derive that \(\eta_\eps(u_n)> z(u_n)\) for every \(n\) large enough. So, by a similar argument to that used above, we get  that for every \(n\) sufficiently large we have 
 \(z(u)\le  \eta_\epsilon  (u)\)  for every \(u\in (u_{n+1},u_n)\). 
Therefore, again we derive for some \(\tilde \delta>0\)
\[ \frac{D(u)}{z(u)} \le \frac{ f(u_0)-cg(u_0)}{\rho(u_0)}+\eps \quad \text{ for every } u\in (u_0,u_0+\tilde \delta).  \]

Summarizing, whatever  it is the sign of \(D\) in a right neighborhood of \(u_0\) we have 
\beq \label{eq:limsup} \limsup_{u\to u_0^+}  \frac{D(u)}{z(u)} \le  \frac{ f(u_0)-cg(u_0)}{\rho(u_0)}.\eeq
 
Let us now study the liminf. In order to do this, replace \(\eps\) with \(-\eps\) in the definition of \(\eta_\eps\), that is consider now the function 

\[
\eta_{-\varepsilon}\left(u\right)=-\frac{\rho\left(u_{0}\right)}{cg\left(u_{0}\right)-f\left(u_{0}\right)+\varepsilon\rho\left(u_{0}\right)}D\left(u\right).
\]
By  arguing in a similar manner as in the previous case, we can show 
that $\eta_{-\varepsilon}\left(u\right)$ is a strict lower-solution
for the equation in  \eqref{eq:pr-h} and we have

\[\label{eq:limun3}
\frac{\eta_{-\eps}(u_n)}{z(u_n)} \to \frac{cg\left(u_{0}\right)-f\left(u_{0}\right)}{cg\left(u_{0}\right)-f\left(u_{0}\right)+\varepsilon\rho\left(u_{0}\right)}<1 \quad \text{ as  } n\to +\infty.
\]
instead of \eqref{eq:limun2}. Following the same  reasoning above developed we achieve that  
\[ \frac{\eta_{-\eps}(u)}{z(u)} <1 \quad \text{ for every } u\in (u_0, u_0+\sigma)\]
for some \(\sigma >0\), whatever it is the sign of \(D\) in a right neighborhhod of \(u_0\). Therefore 
\[
\frac{D\left(u\right)}{z_{c}\left(u\right)}>\frac{f\left(u_{0}\right)-cg\left(u_{0}\right)}{\rho\left(u_{0}\right)}-\varepsilon \quad \text{ for every } u\in (u_0, u_0+\sigma).
\]
Thus, 
\[ \liminf_{u\to u_0^+} \frac{D(u)}{z(u)} \ge \frac{f\left(u_{0}\right)-cg\left(u_{0}\right)}{\rho\left(u_{0}\right)}\]
which jointly with \eqref{eq:limsup} implies that 
\[ \lim_{u\to u_0^+} \frac{D(u)}{z(u)} = \frac{f\left(u_{0}\right)-cg\left(u_{0}\right)}{\rho\left(u_{0}\right)}.\]


\end{proof}

\bg
\begin{rem}\label{r:zc1}{\rm Taking account of Proposition \ref{p:limite} and the differential equation of problem \eqref{eq:pr-h}, we have that 
  if \(z\) is differentiable at some point \(u_0\in D_0\) then there exists the limit \(\dis\lim_{u \to u_0} \dot z(u)\). So, if \(z\) is a solution of problem \eqref{eq:ordine1} with \(\dot z_+(u)=\dot z_-(u)\) for every \(u\in D_0\), then it is a \(C^1\) function in \((0,1)\). }
\end{rem}

\bg
We are now ready to prove Theorem \ref{t:main2}.
\bg

{\em Proof of Theorem \ref{t:main2}.} \ 
By virtue of Theorem \ref{t:gen}, there exists a value \(c^*\) satisfying \eqref{eq:stimagen} such that problem \eqref{eq:ordine1} 
admits a solution \(z_c\) if and only if \(c\ge c^*\).
 Recalling that  \((0,1)\setminus D_0=  \bigcup_{k=1}^n (\alpha_k,\beta_k) \), put
\beq \label{d:c**} \hat c :=\max\left\{c^*, \max_{k\in K_0^-} \frac{f(\alpha_k)}{g(\alpha_k)} \right\}\eeq
(with \(\hat c =c^*\) if \(K_0^-=\emptyset\)). 

Let us
 fix \(c>\hat c \) and let \(z_c\) be a solution of problem \eqref{eq:ordine1} given by Theorem \ref{t:gen}. In force of Theorem \ref{t:equiv} we have to show that the function \(z_c/D\) admits a continuous extension on \((0,1)\), that is there exists finite the limit 
 \[ \lim_{u\to \alpha_k} \frac{z_c(\alpha_k)}{D(\alpha_k)} \quad \text{ for every } k=2,\cdots,n.\]
 
Let us fix an index \(k\ge 2\).
If \(\dot D(\alpha_k)\ne 0\), then  
there exists the limit
\[ \lim_{u \to \alpha_k} \frac{z_c(u)}{D(u)}= \frac{\dot z_c(\alpha_k)}{\dot D(\alpha_k)}.\]

Instead, if \(\dot D(\alpha_k)=0\), we have \(\dot h(\alpha_k)=0\) and 
then by \eqref{ip:nonzero}, \eqref{d:c**} and \eqref{eq:stimagen}  we deduce that \(f(\alpha_k)-cg(\alpha_k)<0\). 
  Therefore, from \eqref{eq:zpuntoh} we have \(\dot z_c(u_0)=0\). Then, by applying  Proposition \ref{p:limite} we derive
\[ \lim_{u \to \alpha_k} \frac{z_c(u)}{D(u)}= \frac{\rho(\alpha_k)}{f(\alpha_k) -cg(\alpha_k)}.\]
Therefore, for every \(c>\hat c \) there exist t.w.s. 

\md
Let su now take \(c<\hat c \). If \(c<c^*\) then the first order singular problem \eqref{eq:ordine1} has no solution, hence, a fortiori,  no t.w.s. exists by Theorem \ref{t:equiv}. Instead, if \(c^*\le c<  \hat c \), let \(z_c\) be a solution of problem \eqref{eq:ordine1} given by Theorem 
\ref{t:gen}.  By \eqref{d:c**} we get \(f(\alpha_k)-cg(\alpha_k)>0\) for some \(k\in K_0^-\). Then, by Proposition \ref{p:limite} we infer \(\dot z(\alpha_k)\ne 0\) and so
\[ \lim_{u \to \alpha_k} \left|\frac{z_c(u)}{D(u)}\right|= +\infty.\]
Then, the function \(z/D\) has not a continuous extension in \((0,1)\) and by virtue of  Theorem \ref{t:equiv} we get that no t.w.s. exists.

\hfill \(\Box\)

\begin{rem}\label{r:cstar-2}
{\em In view of what we observed in Remark \ref{r:cstar}, the existence of t.w.s. is not ensured when \(c=\hat c\). Indeed, since  \(z_{\hat c}\) is always right- and left-differentiable at \(u_0\),  when \(\dot D(u_0)\ne 0\), the limit of \(z_{\hat c}(u)/D(u)\) as \(u\to u_0\) exists finite if and only if there exists \(\dot z_{\hat c}(u_0^+) = \dot z_{\hat c}(u_0^-)\).
\newline
\sm
However, if one allows the equation to admit weak t.w.s. (see \cite{CM1}), then by similar techniques to those used in \cite{CM1} it is possible to show that weak t.w.s. exist also for \(c=\hat c\).  
 }
\end{rem}

Observe now that when \(K^+= \{1\}\), that is when \(D\) has an unique change of sign, from positive to negative, we can state the existence of t.w.s. also for \(c=\hat c\),  as in the following result.

\bg
\begin{cor}
\label{c:one}
Put  \(h(u):=D(u)\rho(u)\), suppose that \(h\) is differentiable at 0,1. Moreover, assume
that for some \(u_0\in (0,1)\) we have  \(D(u)(u-u_0)<0\) for every \(u\in (0,1)\setminus \{u_0\}\).

Moreover, suppose    \( g(0)>0\),  \(g(1)>0 \) and

\[
\int_{0}^u g(s)  \de s >0,\ \text{ for any } 0<u \le u_0; 
\quad \int_u^{1} g(s)  \de s >0,\ \text{ for any } u_0\le u <1. 
\]
Finally, in the case \(\dot D(u_0)=0\)  suppose also that \(g(u_0)>0\) and one of the following assumptions are satisfied:

\beq
\label{ip:u0}
\frac{f(u_0)}{g(u_0)}\ < \ \max\left\{ \frac{f(0)}{g(0)}, \frac{f(1)}{g(1)} \right\}
\eeq

or

\beq \label{ip:f0}
f(u) \equiv 0 \quad \text{ in } (0,u_0)  \quad \text{ or  } \quad f(u) \equiv 0 \quad \text{ in } (u_0,1).
\eeq

Then, there exists a value \(\hat c  \)  such that equation \eqref{eq:E}
admits t.w.s. if and only if  \(c\ge \hat c \).

\md
Moreover, 
we have
\[
\max\left\{ \frac{2\sqrt{\dot h(0)}+f(0)}{g(0)}  \ , \ \frac{2\sqrt{\dot h(1)}+f(1)}{g(1)} \right\} \le  \hat c \le \max\left\{ \frac{2\sqrt{H^+}+F^+}{G^+}  \ , \  \frac{2\sqrt{H^-}+F^-}{G^-}  \right\}.
\]
where

\[
\begin{array}{ll}
G^+:= \dis\inf_{u\in (0,u_0)} -\!\!\!\!\!\!\int_{0}^u g(s)  \de s, \quad & G^-:=\dis\inf_{u\in (u_0,1)} -\!\!\!\!\!\!\int_u^{1} g(s)  \de s \\ F^+:= \dis\sup_{u\in (0,u_0)} -\!\!\!\!\!\!\int_{0}^u f(s)  \de s, \quad & F^-:=\dis\sup_{u\in (u_0,1)} -\!\!\!\!\!\!\int_u^{1} f(s)  \de s \\
H^+:= \dis\sup_{u\in (0,u_0)} -\!\!\!\!\!\!\int_{0}^u \frac{h(s)}{s} \ \de s, \quad & H^-:=\dis\sup_{u\in (u_0,1)} -\!\!\!\!\!\!\int_u^{1} \frac{h(s)}{s-1} \  \de s. 
\end{array}\]

\end{cor}

\begin{proof}

In view of Theorem \ref{t:main2}, we have only to prove that 
there exists a t.w.s. also for \(c=\hat c \),  that is the limit \(\dis\lim_{u\to u_0}  \frac{z_{\hat c}(u)}{D(u)}\) exists finite.

Observe that by  \eqref{eq:puntoz} and \eqref{eq:puntozneg} we have \(\dot z_{\hat c }(u_0)= r_+(\hat c , u_0)\).
So, if \(\dot D(u_0)\ne 0\) we get 
\[
\lim_{u \to u_0} \frac{z_{\hat c }(u)}{D(u)} = \frac{\dot z_{\hat c}(u_0)}{\dot D(u_0)} \in \RR.
\]

Otherwise, if \(\dot D(u_0)=0\) then \(K_0^-\ne \emptyset\) so  \(\hat c = \max \{c^*,  f(u_0)/g(u_0)\}\). Therefore, since \(\dot h(u_0)=0\), by \eqref{eq:puntoz} and \eqref{eq:puntozneg} we have \(\dot z_{\hat c}(u_0)=0\). 

If   \eqref{ip:u0} is satisfied, then by \eqref{eq:stimagen} we have \(f(u_0)-\hat c  g(u_0)<0\), so by applying Proposition \ref{p:limite} we get

\beq \label{eq:limite finito}
\lim_{u \to u_0} \frac{z_{\hat c }(u)}{D(u)} = \frac{\rho(u_0)}{f(u_0)-\hat c g(u_0)} \in \RR.
\eeq

Otherwise, if \eqref{ip:f0} holds, then necessarily \(\hat c>0\). Indeed, if \(\hat c=0\) and \(f(u)\equiv 0\) in \((0,u_0)\) [\((u_0,1)\)], then by the differential equation in \eqref{eq:ordine1} we obtain 
\[ \dot z_{\hat c}(u)=-\frac{\rho(u)D(u)}{z(u)}>0 \quad \text{for every } u\in (0,u_0)\ [(u_0,1)]\]
in contradiction with the prescribed sign of \(z_{\hat c}\) and the boundary condition \(z(0^+)=0\) [\(z(1^-)=0\)]. Therefore, also in this case we have \(f(u_0)-\hat c g(u_0)=-\hat c g(u_0)<0\) and we can apply again Proposition \ref{p:limite} and achieve \eqref{eq:limite finito}.

\md

Summarizing, the function \(z_{\hat c }/D\) admits a continuous extension in \((0,1)\) and there exists t.w.s. also for \(c=\hat c \).

\sm
Finally, as for estimate of \(\hat c\), note that if \(\dot D(u_0)\ne 0\), then \(K_0^-=\emptyset\) and \eqref{eq:stimagen3} becomes the present estimate. Instead, if \(\dot D(u_0)=0\) and \eqref{ip:u0} or \eqref{ip:f0} hold, then again \eqref{eq:stimagen3} becomes the present one.

\end{proof}

\section{Classification of the t.w.s.}

\md

In this Section we provide a classification of the t.w.s. distinguishing between  classical t.w.s. and sharp ones.
Recall that the classification can be done considering the maximal existence interval \((a,b)\): the t.w.s. are classical if \(a=-\infty,\ b=+\infty\); sharp of type 1 if \(a=-\infty, \ b\in \RR\), sharp of type 2  if \(a\in \RR, \ b=+\infty\), and finally sharp of type 3 if both \(a,b\in \RR\) (see Section \ref{s:preliminary}).  Moreover, the extrema \(a, b\) are finite or infinite according to the values \(u'(a^+)\) and \(u'(b^-)\), that is if these are negative or null. Therefore, in view of Theorem \ref{t:equiv}, we  need to know the behavior of the function \(z_c(u)/D(u)\) as \(u\to 0\) and \(u\to 1\). This is not trivial when \(D(0)=\dot D(0)=\dot z(0)=0\) or  \(D(1)=\dot D(1)=\dot z(1)=0\). The following result concerns just this topic.

\bg

\begin{prop}\label{p:limite0}
Let all the assumptions of Theorem \ref{t:gen} be satisfied, for \(h(u):=D(u)\rho(u)\). Let \(z\) be a solution of problem \eqref{eq:ordine1} for some  \(c\ge c^*\).

 Then, if \( \dot z(0)=D(0)=\dot D(0)=0\) 
we necessarily have \(f(0)-cg(0)\le 0\). 

Moreover, if \(f(0)-cg(0)<0\) then 
\[ \lim_{u\to 0^+}  \ \frac{z(u)}{D(u)} = 0.\] 

Similarly, if  \( \dot z(1)=D(1)=\dot D(1)=0\) 
we necessarily have \(f(1)-cg(1)\le 0\). 

Moreover, if \(f(1)-cg(1)<0\) then 
\[ \lim_{u\to 1^-}  \ \frac{z(u)}{D(u)} = 0.\]

\end{prop}

\begin{proof}
Let us prove the assertion concerning the limit as \(u\to 0^+\).

\md
Let us first consider the case \(D(u)>0\) in a right neighborhood  \(0\). Then,  necessarily \(f(0)-cg(0)\le 0\). Indeed, otherwise, from the differential equation in \eqref{eq:ordine1} we would have 
\(\dot z(u)>0\) in \((0,\delta)\) for some \(\delta>0\), in contradiction with the prescribed negative sign of \(z\) in a right neighborhood of \(0\).

So, assume now \(f(0)-cg(0)< 0\).
Let us fix \(\varepsilon>0\) and put \(\psi_\varepsilon(u):= -\varepsilon  D(u)\). Then, we have
\[ \lim_{u\to 0^+} \left( f(u)-cg(u) -\frac{D(u)\rho(u)}{\psi_\varepsilon (u)}\right) =f(0)-cg(0)<0. \]
Moreover, since \(\dot \psi_\varepsilon (u)=-\varepsilon \dot D(u)\to 0\) as \(u\to 0^+\), there exists a value \(\delta=\delta_\varepsilon>0\) such that 
\beq \label{eq:sopra}
\dot \psi_\varepsilon(u)> f(u)-cg(u) -\frac{D(u)\rho(u)}{\psi_\varepsilon (u)} \quad \text{ whenever } 0<u<\delta 
\eeq
So, \(\psi_\varepsilon\) is an upper-solution.

Since
 \(\dot z(0)=0\), we can choose a sequence \(u_n \to 0^+\) such that \(\dot z(u_n)\to 0\), so
\[ \frac{z(u_n)}{D(u_n)}=\frac{\rho(u_n)}{f(u_n)-cg(u_n)-\dot z(u_n)} \to 0 \]
and then 
  for a natural \(n^*\) sufficiently large we have
\[ z(u_{n^*})> \psi_\varepsilon(u_{n^*}).\]

 Therefore, taking \eqref{eq:sopra} into account, we derive \(z(u)> \psi_\varepsilon (u)=-\varepsilon D(u)\)
for every \(u\in (0,u_{n^*}]\) (see Lemma \ref{l:Gronwall})  and then 
\[ -\varepsilon < \frac{z(u)}{D(u)}< 0 \quad \text{ for every } u\in (0,u_{n^*}].\]
that is  the assertion.

\md
Assume now \(D(u)<0\) in a right neighborhood of \(0\).
In this case, by \eqref{eq:puntozneg} we have (see also \eqref{d:rpm}) \(\dot z(0)=\max\{f(0)-cg(0),0\}\), so being \(\dot z(0)=0\), we again infer \(f(0)-cg(0)\le 0\).

Moreover, if \(f(0)-cg(0)< 0\), the function \(\psi_\varepsilon\) above defined is again a strict upper-solution in a right neighborhood of \(0\). So, being \(z(0)=\psi_\varepsilon(0)=0\) we can apply again Lemma \ref{l:Gronwall} to obtain 
\(z(u)<\psi_\varepsilon(u)\), that is \(\dis \frac{z(u)}{D(u)}>-\varepsilon\) in the same right neighborhood and again we deduce the assertion.

\md
The proof regarding the limit as \(u\to 1^-\) can be derived by the change of variable considered in Lemma \ref{l:negative}.

\end{proof}

\md 
We now have all the tools to classify the t.w.s. 
From now on we will assume there exist, finite or not, the limits

\[ 
\lim_{u\to 0^+} \dot D(u):=\dot D(0^+), \qquad \lim_{u\to 1^-} \dot D(u):=\dot D(1^-).
\]
Moreover, assume there exist finite
\beq
\label{ip:doth}
\lim_{u\to 0^+} \frac{D(u)\rho(u)}{u}:=\ell_0, \qquad \lim_{u\to 1^-} \frac{D(u)\rho(u)}{u-1}:=\ell_1.
\eeq

\md

The following result provides a classification for the t.w.s. having speed \(c>\hat c\).

\md
\begin{prop}\label{p:class}  Under the same assumptions of Theorem \ref{t:main2}, let 
 \(u\) be a t.w.s. of equation
\[ \left(D\left(u\right(t))u^{\prime}(t)\right)^{\prime}+\left(cg\left(u(t)\right)-f(u(t))\right)u^{\prime}(t)+\rho(u(t))=0\] 
defined in its maximal existence interval \((a,b)\), for some \(c>\hat c \), where \(\hat c \) is the threshold wave speed given by Theorem \ref{t:main2}.

Then,
\begin{itemize}
\item if \(D(u)>0\) in \((0,\delta)\) for some \(\delta>0\),  then \(b=+\infty\);

\item if \(D(u)<0\) in \((0,\delta)\) for some \(\delta>0\) and 
the three values \(f(0)-cg(0)\), \(D(0)\), \(\dot D(0)\) do not vanish simultaneously, then 
\[ b\in \RR \quad \text{ if and only if }  \quad D(0)=0, \ \dot D(0^+)>-\infty \ \text{ and } \ f(0)-cg(0)>0.\]

\end{itemize}

\md
Moreover, 

\begin{itemize}
\item if \(D(u)<0\)   in \((1-\delta,1)\) for some \(\delta>0\), then \(a=-\infty\);

\item if \(D(u)>0\)  in \((1-\delta,1)\) for some \(\delta>0\) and 
the three values \(f(1)-cg(1)\), \(D(1)\), \(\dot D(1)\) do not vanish simultaneously, 
\[ a\in \RR \quad \text{ if and only if }  \quad D(1)=0, \ \dot D(1^-)>-\infty \ \text{ and } \ f(1)-cg(1)>0.\]

\end{itemize}

\end{prop}

\md
\begin{proof}
\md
First of all, notice that
 by \eqref{ip:doth} we have that \(z\) is differentiable at \(0\) and \(1\) as a consequence of Theorems \ref{t:pos} and \ref{t:neg}. Moreover, 
 by \eqref{eq:stimagen} for every \(c>c^*\) we have 
\beq
\label{eq:climit1}
f(\alpha_k)+2\sqrt{\dot h(\alpha_k)}-cg(\alpha_k)<0  \quad \text{ for every } k\in K^+
\eeq
and
\[
f(\beta_k)+2\sqrt{\dot h(\beta_k)}-cg(\beta_k)<0  \quad \text{ for every } k\in K^-.
\]

\md
Of course, if  \(D(0)\ne 0\) then  \(z(u)/D(u)\to 0\) as \(u\to 0^+\) and this is true also when  \(D(0)=0\) and \(\dot D(0^+)=\pm \infty\), owing to  the differentiability of \(z\). 
 So, from now on let us assume \(D(0)=0\) and \(\dot D(0^+)\) is finite (hence \(\dot D(0^+)=\dot D(0)\)). In this case, in \eqref{ip:doth} we have \(\ell_0=0\) and from 
 \eqref{eq:zpuntoh} we deduce
\beq\label{eq:zpunto0} \dot z(0)= \max\{0, f(0)-cg(0)\}.
\eeq 
 
\sm
We now split the proof into two cases.

\md
\(\bullet\) \ {\em Case 1}: \(D\) is positive in a right neighborhood of \(0\).

\md
In this case,  by \eqref{eq:climit1} we have
\beq \label{eq:limit0}
f(0) - cg(0)<0.
\eeq

 Then, by  \eqref{eq:zpunto0} 
 we have \(\dot z(0)= 0\). 
 
 Hence, if \(\dot D(0)\ne 0\) 
 there exists the limit \(\dis\lim_{u \to 0^+} \frac{z(u)}{D(u)}=\frac{\dot z(0)}{\dot D(0)}=0\). Instead, if \(\dot D(0)=0\) then by \eqref{eq:limit0} we can apply Proposition \ref{p:limite0} to infer \(\dis\lim_{u \to 0^+} \frac{z(u)}{D(u)}=0\).

Summarizing, when \(D\) is positive in a right neighborhood of \(0\) and \(c>\hat c\) we have \(\dis\lim_{t\to b^-}u'(t)= \dis\lim_{u \to 0^+} \frac{z(u)}{D(u)}=0\) and this means that \(b=+\infty\) (since \((a,b)\) is the maximal existence interval of the solution).

\md
\(\bullet\) \ {\em Case 2}: \(D\) is negative in a right neighborhood of \(0\).

\md
First consider the case \(f(0)-cg(0)>0\). By \eqref{eq:zpunto0} we have \(\dot z(0)=f(0)-cg(0)\). Then we have
\[ \lim_{u \to 0^+} \dfrac{z(u)}{D(u)}  = \begin{cases} \dfrac{f(0)-cg(0)}{\dot D(0)} & \text{ if } \dot D(0) \ne 0 \\ -\infty & \text{ if } \dot D(0)=0. \end{cases}\]
Therefore, when \(f(0)-cg(0)>0\) we have \(\dis \lim_{u \to 0^+} \dfrac{z(u)}{D(u)} <0\), implying that \(b<+\infty\).

Instead, if \(f(0)-cg(0)\le 0\) then \(\dot z(0)=0\) and 
\[ \lim_{u \to 0^+} \dfrac{z(u)}{D(u)}=0 \quad \text{whenever } \dot D(0)\ne 0.\]
So, we get \(b=+\infty\) if \(f(0)-cg(0)\le 0\) and \(\dot D(0)\ne 0\).

\sm
Finally, let us consider the last case
 \(f(0)-cg(0)<0\) and \(\dot D(0)=0\). By \eqref{eq:zpunto0} we infer that \(\dot z(0)=0\) and by applying Proposition \eqref{p:limite0} we achieve \(\dis \lim_{u \to 0^+} \dfrac{z(u)}{D(u)} =0\), implying again that \(b=+\infty\).
 
\sm
This concludes the proof of the statement regarding the extremum \(b\).

\md The proof concerning the extremum \(a\) can be derived by means of the change of variable given by Lemma \ref{l:negative}. 
 
 \end{proof}

\md
\begin{rem}\label{r:c=}
{\em If equation \eqref{eq:E} admits t.w.s. for \(c=\hat c\) (see Remarks \ref{r:cstar}, \ref{r:cstar-2} and Corollary \ref{c:one}), then it is  possible to classify the t.w.s. also for \(c=\hat c\).
We avoid to give here a complete classification for \(c=\hat c\)  since it involves the subdivision into many different cases. However, it is possible to 
classify the t.w.s. (according to the specific situation under investigation), reasoning as in Proposition \ref{p:class} by applying Proposition \eqref{p:limite0}  
when \(D(0)=\dot D(0)=\dot z(0)=0\), which is the only non-trivial case.}
\end{rem} 
 
\begin{rem}\label{r:gneg}{\em
Similarly to what we observed in \cite[Remark 4]{CaMaPa}, if \(g\) is "predominantly" negative, we can verify if  the opposite function \(-g\) satisfies the assumptions of the present results. If it does, then we can replace \(c\) with \(-c\) and obtain the existence of a threshold value \(c^{**}\) such that the t.w.s. exists for \(c<c^{**}\) and they do not exist for \(c>c^{**}\) and deduce all the other results too.}
 \end{rem}

\section{Examples}
In this section we present some examples to which our general results can be applied.

\bigskip
{\bf Example 1.} 
Consider the equation 
\[
\left(u^2-u +\mathcal{ K}\right)u_{\tau}=\left(\left(3/4-u\right)\sqrt{u-u^2}\ u_{x}\right)_{x}+\sqrt{u-u^2}
\]
in which
\[
g(u):=u^2-u+\mathcal{ K}\ , \quad
f\left(u\right)\equiv0
\]
\[
D(u):=\left(3/4-u\right)\sqrt{u-u^2} \ , \quad
\rho\left(u\right):=\sqrt{u-u^2}
\]
where $\mathcal{ K}>3/16$ is a fixed number. Of course, assumptions \eqref{ip:funz}, \eqref{eq:sign rho} and \eqref{ip:Dfinite} are satisfied. Moreover, 
taking $h(u):=D(u)\rho(u)= u(1-u)(\frac34-u)$, it is  continuous
in $\left[0,1\right]$ and differentiable at $0,1$. Moreover,  $h(u)$ is positive
if $0<u<3/4$, negative if $3/4<u<1$. 

Note that  $g(0)=g(1)=\mathcal{ K}>0$ and 
\[ \dis\int_{0}^{u}g(s) \de s=\frac16u(2u^2-3u+6\mathcal{ K})>0 \quad \text{
for all }  u\in\left[0,3/4\right]\]
since $\mathcal{ K}>\frac{3}{16}$. Moreover,  $\dis \int_{u}^{1}g(s)\de s>0$ for all $u\in\left[3/4,1\right]$, since $g$ is positive in this interval.
However, when \(\frac{3}{16}< \mathcal{ K}< \frac14\), \(g\) assumes also negative values. 

Finally,
we observe that $D_{00}$ and $K_{0}^{-}$ are empty. 

So, by Corollary \ref{c:one}, we can conclude
that there exists $\hat{c}$ such that equation \eqref{eq:E} admits t.w.s.
if and only if $c\geq \hat{c}$. Furthermore, we have 
\(\dot h(0)=\frac{3}{4},\, \dot h(1)= \frac{1}{4}\), and by simple computations one can verify that
\[
\sup_{u\in\left(3/4,1\right)}-\!\!\!\!\!\!\int_{u}^{1}\frac{h(s)}{s-1}\de s=\frac{1}{4}, \quad \sup_{u\in\left(0,3/4\right)}-\!\!\!\!\!\!\int_{0}^{u}\frac{h(s)}{s}\de s=\frac{3}{4},
\]
\[
\inf_{u\in\left(3/4,1\right)}-\!\!\!\!\!\!\int_{u}^{1}g(s)\de s=\mathcal{ K}-\frac{5}{48},\quad \inf_{u\in\left(0,3/4\right)}-\!\!\!\!\!\!\int_{0}^{u}g(s)\de s=\mathcal{ K}-\frac{3}{16}.
\]
So, we achieve the following estimate for \(\hat c\):
\[\frac{\sqrt{3}}{\mathcal{K}}\leq\hat{c}\leq\frac{\sqrt{3}}{\mathcal{K}-\frac{3}{16}}.
\]
Finally, concerning to the classification of the t.w.s. we have that they are classical for every \(c> \hat c\) by Proposition \ref{p:class} and they are classical also for \(c=\hat c\) since \(\dot D(0)=\dot D(1)=+\infty\) (see Remark  \ref{r:c=}).

\bg
{\bf Example 2.} 
Consider the equation 
\[
\left(u^2-u+\mathcal{ K}\right)u_{\tau}=\left(\left(1/2-u\right)\left(u-u^2\right)^\alpha u_{x}\right)_{x}+(u-u^2)^\beta
\]
where \(\alpha,\beta>0\) with \(\alpha+\beta\ge 1\), \(\mathcal{K}>\frac16\), in which
\[
g(u):=u^2-u+\mathcal{ K}\ , \quad
f\left(u\right)\equiv0
\]
\[
D(u):=\left(1/2-u\right)(u-u^2)^\alpha \ , \quad
\rho\left(u\right):=(u-u^2)^\beta.
\]
Also in this case one can easily verify that all the assumptions of Corollary \ref{c:one} are satisfied, so there exists a threshold value \(\hat c\) such that t.w.s. exist if and only if \(c\ge \hat c\) and they are classical for every \(c>\hat c\).

Instead, regarding the classification of the t.w.s. for \(c=\hat c\), note that since \(K_0^-=\emptyset\), by \eqref{d:c**} we have \(\hat c= \max\{c_1^*, c_2^*\}\)
where \(c_1^*\) [resp. \(c_2^*\)] is the threshold value of the admissible wave speeds in the interval \((0,\frac12)\) [\((\frac12,u)\)] given by Theorem \ref{t:pos} [Theorem \ref{t:neg}].
So, if \(\hat c=c_1^*\) we have \(\dot z(0)=-\hat c g(0)=-\mathcal K \hat c<0\) and consequently the corresponding t.w.s. reaches the equilibrium \(0\) at a finite time (t.w.s. sharp of type (1) or type (3)). Similarly, if  if \(\hat c=c_2^*\) we have \(\dot z(1)=-\hat c g(1)=-\mathcal K \hat c<0\) and consequently the corresponding t.w.s. attains the value of the equilibrium \(1\) at a finite time (t.w.s. sharp of type (2) or type (3)).

Hence, the t.w.s. for \(c=\hat c\) is sharp, but 
since the values \(c_1^*\) and \(c_2^*\) are unknown in general, we can not determine what kind of sharp t.w.s. it is.

\bg
{\bf Example 3.} 
Consider the equation 
\[
u_{\tau}+ u_x=\left(\left(1/2-u\right)^2 u_{x}\right)_{x}+(u-u^2)
\]
 in which
\[
g(u):\equiv 1\ , \quad
f\left(u\right)\equiv 1
\]
\[
D(u):=\left(1/2-u\right)^2 \ , \quad
\rho\left(u\right):=(u-u^2).
\]

In this case we have \(D(u)\ge 0\) for every \(u\in (0,1)\), but \(D(\frac12)=0\). So, we have \(K^+=\{1,2\}\), \(K^-=\emptyset.\) Then, by Theorem \ref{t:main2} we have that there exists a value \(\hat c\) such that the equation adimts t.w.s. for every \(c>\hat c\), and no t.w.s. exists for \(c<\hat c\). Moreover, by Proposition \ref{p:class} we have that the t.w.s are classical for every \(c>\hat c\).

Furthermore, note that in this case we can obtain a precise value for \(\hat c\). Indeed, \(\dot h(0)=\frac14\), \(\dot h(\frac12)=0\), \(H_1^+=\frac14\) and \(H_2^+=\frac{1}{12 \sqrt 6}\). So, by \eqref{eq:stimagen3} we obtain \(\hat c=2\).

\sm
Finally, no t.w.s. exists for \(c=\hat c\), since \(\dot D(\frac12)=0\) but by \eqref{eq:puntoz} we have \(\dot z(\frac12^-)=\frac12\) and \(\dot z(\frac12^+)=-\frac32\). So, the function \(z/D\) has not a continuous extension in the interval \((0.1)\).

\subsection*{Funding Information}
This article was supported by PRIN 2022 -- Progetti di Ricerca di rilevante Interesse Nazionale, ``Nonlinear differential problems with applications to real phenomena'' (2022ZXZTN2). The third author is partially supported by INdAM--GNAMPA project ``Problemi ellittici e sub--ellittici: non linearit\`a, singolarit\`a e crescita critica'' CUP: E53C23001670001.

\subsection*{Acknowledgements}
All three authors are members of the Gruppo Nazionale per l'Analisi Matematica, la Probabilit\`a e le loro Applicazioni (GNAMPA) of the Istituto Nazionale di Alta Matematica (INdAM).

\end{document}